\documentclass[12pt]{article}
\usepackage{amssymb,amsmath}
\usepackage{amsbsy}
\usepackage{amsthm}
\usepackage{multirow}
\usepackage{mathrsfs}
\usepackage{verbatim}
\usepackage[colorlinks]{hyperref}
\usepackage{fullpage}
\usepackage{mathdots}
\usepackage{graphicx,subfigure}
\usepackage[english]{babel}
\usepackage[utf8]{inputenc}
\usepackage{tikz}
\usepackage{mathtools}

\usepackage{authblk}

\usetikzlibrary{backgrounds,fit, matrix}
\usetikzlibrary{positioning}
\usetikzlibrary{calc,through,chains}
\usetikzlibrary{arrows,shapes,snakes,automata, petri}

\newtheorem{Theorem}[equation]{Theorem}
\newtheorem{Corollary}[equation]{Corollary}

\newtheorem{Proposition}[equation]{Proposition}

\newtheorem{Remark}[equation]{Remark}

\theoremstyle{definition}
\newtheorem{Definition}[equation]{Definition}
\newtheorem{Example}[equation]{Example}

\numberwithin{equation}{section}
\numberwithin{figure}{section}

\newcommand{\C}{{\mathbb C}}
\newcommand{\Z}{{\mathbb Z}}

\newcommand{\N}{{\mathbb N}}

\newcommand{\si}{I_{p,q}^\pm}

\newcommand{\mc}[1]{\mathcal{#1}}
\newcommand{\ms}[1]{\mathscr{#1}}

\newcommand{\mt}[1]{\text{#1}}

\begin{document}

\title{The Genesis of Involutions\\
(Polarizations and Lattice Paths)}

\author[1]{Mahir Bilen Can}
\author[2]{\"Ozlem U\u{g}urlu}

\affil[1]{{\small Tulane University, New Orleans; mahirbilencan@gmail.com}}    
\affil[2]{{\small Tulane University, New Orleans; ougurlu@tulane.edu}}

\normalsize

\date{\today}
\maketitle

\begin{abstract}
The number of Borel orbits in polarizations 
(the symmetric variety $SL_{n}/S(GL_p\times GL_q)$) is analyzed, 
various (bivariate) generating functions are found. 
Relations to lattice path combinatorics are explored. 
\vspace{.2cm}

\end{abstract}

\section{Introduction}\label{S:Introduction}

The classification of automorphisms of order two on a simple 
algebraic group $G$ has a long history. Over the field of 
complex numbers $k=\C$, this problem is equivalent to 
classifying the real forms of simple complex Lie algebras 
and it is solved by E. Cartan. Subsequently, 
by the works of Araki~\cite{Araki} and Helminck~\cite{Helminck}, 
Cartan's ``theory of involutions'' is extended over arbitrary fields 
(see~\cite{Springer}, also). According to Cartan's classification, roughly 
stated, in type $A$, there are three nontrivial automorphisms 
of order two up to inner automorphisms;
\begin{enumerate}
\item $G=SL_n$ and $\theta_1 : G\rightarrow G$ is defined by 
$$
\theta_1( g) = (g^{-1})^\top,
$$ 
where $\top$ stands for the matrix-transposition.

\item $G=SL_{2n}$ and $\theta_2 : G\rightarrow G$ is defined by 
$$
\theta_2( g) =J (g^{-1})^\top J,
$$ 
where $J$ denotes the skew form 
$J = \begin{bmatrix} 0 & id_n \\  -id_n & 0 \end{bmatrix}$, 
and $id_n$ is the identity matrix of order $n$. 

\item The third automorphism on $G=SL_n$ is actually a member of a 
family of automorphisms defined as follows: 
Let $p$ and $q$ be two positive integers such that $n=p+q$ and define 
$
J_{p,q} := 
\begin{bmatrix}
0 & 0 & s_q \\
0 & id_{p-q} & 0 \\
s_q & 0 & 0 
\end{bmatrix},
$
where $s_q$ is the $q\times q$ block-matrix with 1's along its anti-diagonal and 
0's elsewhere. The automorphism $\theta_3: SL_n \rightarrow SL_n$, 
which depends on $p$ and $q$, is defined by 
$$
\theta_3 (g ) = J_{p,q} (g^{-1})^\top J_{p,q}.
$$ 
\end{enumerate}

The fixed point subgroup of an automorphism of order two on $G$ is called 
a symmetric subgroup. In type $A$, accordingly, 
the symmetric subgroups are given by 
\begin{enumerate}
\item the special orthogonal group, denoted by $SO_n$, 
\item the symplectic group, denoted by $Sp_n$,
\item the maximal Levi subgroup $S(GL_p\times GL_q)$, which consists of 
block-diagonal matrices of the form 
$\begin{bmatrix}
A & 0 \\ 
0 & B
\end{bmatrix}$, where $A\in GL_p, B\in GL_q$ and $\det A\cdot  \det B = 1$. 
\end{enumerate}
Let $\ms{B}_n$ denote the flag variety of $SL_n$, namely, 
$\ms{B}_n:= SL_n/B_n$,  
where $B_n$ is the Borel subgroup consisting of upper triangular 
matrices in $SL_n$. 
Our goal in this article is to gain a combinatorial 
understanding of the orbits of the symmetric subgroup 
$S(GL_p\times GL_q)$ in $\ms{B}_n$ by 
studying various generating functions and lattice paths. 
To motivate our discussion and to set up some notation, 
let us give a brief account of 
what is known about the actual sets of combinatorial objects that 
parametrize the symmetric subgroup orbits (in type $A$).

\vspace{.5cm}

Let $T_n$ denote the maximal torus
consisting of diagonal matrices in $SL_n$. 
The Weyl group of the pair $(SL_{n},T_n)$, 
denoted by $S_n$, is the quotient 
group $N_{SL_{n}}(T_n)/T_n$, 
where $N_{SL_{n}}(T_n)$ is the 
normalizer subgroup of $T_n$ in $SL_{n}$. 
We will refer to $S_n$ as the {\em permutation group}
since it is naturally isomorphic to the group of 
permutations of the set $\{1,\dots, n\}$.
We will refer to the following triangular decomposition 
of matrices as the {\em Bruhat-Chevalley decomposition}: 
$$
SL_n = \cup_{\sigma \in S_n} B_n \dot{\sigma} B_n.
$$
The dot on $\sigma$ indicates that we are using a coset representative
of $T_n$ from $N_{T_n}(SL_n)$.

The special linear group $SL_n$, hence, any of its subgroups 
act on $\ms{B}_n$ by left multiplication,
$$
SL_n \times \ms{B}_n \rightarrow \ms{B}_n,\qquad (g,aB_n/B_n) \mapsto ga B_n/B_n.
$$
We call the Zariski closure of a $B_n$-orbit in $\ms{B}_n$ a Schubert variety.
On one hand, the Bruhat-Chevalley decomposition for $SL_n$ 
shows that the Schubert varieties are parametrized by the permutation  
group $S_n$. 
On the other hand, the automorphisms $\theta_i$ ($i=1,2,3$) of $SL_{n}$ 
give (set) automorphisms on $S_n$. 
In particular, the fixed point subsets in $S_n$ 
of these automorphisms are given by 
\begin{enumerate}
\item $I_n= \{\pi \in S_n:\ \pi^2 = id \}$, the set of all involutions in $S_n$;
\item $FI_n=\text{ the set of fixed point free involutions in $S_n$}$; 
\item $\si= \text{ the set of ``signed $(p,q)$-involutions''.}$
\end{enumerate}
\begin{Definition} 
A signed $(p,q)$-involution $\pi \in S_n$ is an involution with an 
assignment of $+$ and $-$ signs to the fixed points of $\pi$ such 
that there are $p - q $ more $+$'s than $-$'s if $q \leq p$. 
(If $p \leq q$, there are $q-p$ more $-$ signs than $+$ signs.) 
The set of signed $(p,q)$-involutions is denoted by $\si$
and its cardinality is denoted by $\alpha_{p,q}$. 
\end{Definition}
For example, $\pi = (1\; 6)(2\; 3)(4^+)(5^-)(7^+)(8^+)$ is a 
signed $(5,3)$-involution. In this example, $p=5$ is 
the number of fixed points with a $+$ sign plus the number of 
2-cycles in $\pi$, while $q=3$ is equal to the number of fixed 
points with a $-$ sign plus the number of 2-cycles in $\pi$. 
Indeed, there are $p-q=2$ more $+$'s than $-$'s. 
\vspace{.5cm}

The sets $I_n$ and $FI_n$, respectively, 
parametrize the orbits of $SO_n$ in 
$\ms{B}_n$ and $Sp_n$ in $\ms{B}_{2n}$. 
See~\cite{RS}. The set $\si$ parametrizes 
the $S(GL_p\times GL_q)$-orbits in $\ms{B}_n$. See~\cite{MO,Wyser}
and the preliminaries section of~\cite{CU}.
The numbers of elements of the sets $I_n$ and $FI_n$ are 
well-investigated and many combinatorial properties of them 
found their ways into textbooks. 
See, for example, Bona's~\cite{Bona}. 
We will review some of the pertaining results in the sequel.
The main goal of our article is to give a count of the signed 
$(p,q)$-involutions and to reinterpret them in the context of 
lattice paths. To this end, by partitioning the set $\si$ according to 
the number of 2-cycles that appear in its elements, we define 
the numbers $\gamma_{k,p,q}$ as follows:
$$
\gamma_{k,p,q}:= \text{ the number of signed 
$(p,q)$-involutions with exactly $k$ 2-cycles.}
$$ 
Our first main result is a recurrence relation for $\gamma_{k,p,q}$'s.
\begin{Theorem}\label{T:first main result:intro} 
Let $p$ and $q$ be two positive integers. If $k=0$, 
then we have $\gamma_{0,p,q}= {p+q \choose q}$.
If $1\leq k \leq q$, 
then the following recurrence relation holds true: 
\begin{equation}\label{recurrence}
\gamma_{k,p,q} =  \gamma_{k,p-1,q} +  \gamma_{k,p,q-1} + (q+p-1) \gamma_{k-1,p-1,q-1}.
\end{equation}
Furthermore, the formula for $\gamma_{k,p,q}$ is given by 
\begin{align}\label{A:first formula:intro}
\gamma_{k,p,q} =  \frac{ (p+q)! }{ (p-k)!(q-k)! }   \frac{ (2k)! }{ 2^k k! }.
\end{align}
\end{Theorem}
As a consequence of Theorem~\ref{T:first main result:intro}, 
we derive a 3-term recurrence for $\alpha_{p,q}$'s.
\begin{Corollary}\label{C:first main result:intro} 
If $p$ and $q$ are two positive integers, then we have 
$\alpha_{p,0}=\alpha_{0,q} = 1$. 
Furthermore, in this case, the following recurrence relation holds true:
\begin{align}\label{A:ourrec:intro}
\alpha_{p,q} = \alpha_{p-1,q} + \alpha_{p,q-1} + (p+q-1) \alpha_{p-1,q-1}. 
\end{align}
\end{Corollary}

Our recurrence relation (\ref{A:ourrec:intro})
has a variable coefficient and this complicates the computation of the generating
series for $\alpha_{p,q}$'s. However, by translating 
(\ref{A:ourrec:intro}) into a partial differential equation, 
we found the following beautiful closed form of their generating series.
\begin{Theorem}\label{T:gs:intro}
Let $(r,s)$ denote a pair of variables that is related to $(x,y)$ by the 
following coordinate transformations:
\begin{align}\label{R:reversible:intro}
x(r,s)= \frac{r}{rs+1} \qquad \text{and} \qquad y(r,s)=\frac{rs^2 - 2rs + 2s }{2(rs+1)}.
\end{align}
If $v(x,y)$ denotes the function that is represented by the series 
$\sum_{p,q\geq 0} \alpha_{p,q} x^q \frac{y^p}{p!}$ in a sufficiently 
small neighborhood of the origin $(0,0)$, then in the transformed 
coordinates it is equal to the following function:
\begin{align*}
v(r,s) =\frac{ e^{s} (rs+1)}{1-r}.
\end{align*}
\end{Theorem}

The relations in (\ref{R:reversible:intro}) are invertible, therefore, by 
back substitution, we obtain a bivariate generating series for $\alpha_{p,q}$'s.

\begin{Corollary}\label{C:genfun:intro}
The bivariate generating series for $\alpha_{p,q}$'s is given by 
\begin{align}\label{A:substitute:intro}
\sum_{p,q\geq 0} \alpha_{p,q} x^q \frac{y^p}{p!}
&=\frac{e^{\frac{1-x-\sqrt{x^2-2xy-2x+1}}{x}} (x-\sqrt{x^2-2xy-2x+1})}{-x^2+2xy+2x-1+x\sqrt{x^2-2xy-2x+1}}.
\end{align}
\end{Corollary}

\vspace{.5cm}

Next, we will discuss the relationship between 
$S(GL_p\times GL_q)$-orbits in $\ms{B}_n$ and certain weighted lattice paths.
The precursor of our development is the well-known relationship between the 
lattice paths in a $p\times q$-grid and the $B_n$-orbits in the 
Grassmann variety of $p$-dimensional subspaces in an $n$-dimensional 
complex vector space. More detailed explanation of this classical result on
Schubert varieties will be given in the preliminaries section. 
See Remark~\ref{R:GrassmanPaths}.

The {\em $(p,q)$-th Delannoy number}, denoted by $D(p,q)$, 
is defined via the recurrence relation
\begin{align}\label{D:basic recurrence}
D(p,q)=D(p-1,q)+D(p,q-1)+D(p-1,q-1)
\end{align}
with respect to the initial conditions $D(p,0)=D(0,q)=D(0,0)=1$. 
Loosely speaking, these numbers give a 
count of the lattice paths that move with 
unit east step, $E:=(1,0)$; unit north step $N:=(0,1)$; 
and the diagonal step $D:=(1,1)$.
More precisely, $D(p,q)$ is the number of lattice paths
that start at the origin $(0,0)\in \N^2$ and end at $(p,q)\in \N^2$
moving with $E,N$, and $D$ steps only. 
We will refer to such paths as {\em $(p,q)$ Delannoy paths},
and denote their collection by $\mc{D}(p,q)$.
For example, if $(p,q)=(2,2)$, then $D(2,2)=13$.
In Figure~\ref{F:22}, we depicted the elements of $\mc{D}(2,2)$.

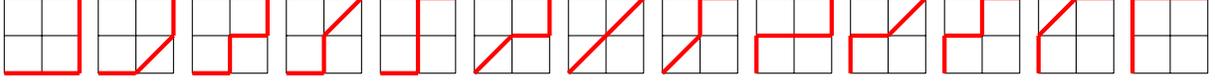
\begin{figure}[h]
\begin{center}
\begin{tikzpicture}[scale=.5]

\begin{scope}[xshift=-15cm]
\draw (0,0) grid (2,2);
\draw[ultra thick,red] (0,0) to (2,0);
\draw[ultra thick,red] (2,0) to (2,2);
\end{scope}

\begin{scope}[xshift=-12.5cm]
\draw (0,0) grid (2,2);
\draw[ultra thick,red] (0,0) to (1,0);
\draw[ultra thick,red] (1,0) to (2,1);
\draw[ultra thick,red] (2,1) to (2,2);
\end{scope}

\begin{scope}[xshift=-10cm]
\draw (0,0) grid (2,2);
\draw[ultra thick,red] (0,0) to (1,0);
\draw[ultra thick,red] (1,0) to (1,1);
\draw[ultra thick,red] (1,1) to (2,1);
\draw[ultra thick,red] (2,1) to (2,2);
\end{scope}

\begin{scope}[xshift=-7.5cm]
\draw (0,0) grid (2,2);
\draw[ultra thick,red] (0,0) to (1,0);
\draw[ultra thick,red] (1,0) to (1,1);
\draw[ultra thick,red] (1,1) to (2,2);
\end{scope}

\begin{scope}[xshift=-5cm]
\draw (0,0) grid (2,2);
\draw[ultra thick,red] (0,0) to (1,0);
\draw[ultra thick,red] (1,0) to (1,1);
\draw[ultra thick,red] (1,1) to (1,2);
\draw[ultra thick,red] (1,2) to (2,2);
\end{scope}

\begin{scope}[xshift=-2.5cm]
\draw (0,0) grid (2,2);
\draw[ultra thick,red] (0,0) to (1,1);
\draw[ultra thick,red] (1,1) to (2,1);
\draw[ultra thick,red] (2,1) to (2,2);
\end{scope}

\begin{scope}[xshift=0cm]
\draw (0,0) grid (2,2);
\draw[ultra thick,red] (0,0) to (2,2);
\end{scope}

\begin{scope}[xshift=2.5cm]
\draw (0,0) grid (2,2);
\draw[ultra thick,red] (0,0) to (1,1);
\draw[ultra thick,red] (1,1) to (1,2);
\draw[ultra thick,red] (1,2) to (2,2);
\end{scope}

\begin{scope}[xshift=5cm]
\draw (0,0) grid (2,2);
\draw[ultra thick,red] (0,0) to (0,1);
\draw[ultra thick,red] (0,1) to (2,1);
\draw[ultra thick,red] (2,1) to (2,2);
\end{scope}

\begin{scope}[xshift=7.5cm]
\draw (0,0) grid (2,2);
\draw[ultra thick,red] (0,0) to (0,1);
\draw[ultra thick,red] (0,1) to (1,1);
\draw[ultra thick,red] (1,1) to (2,2);
\end{scope}

\begin{scope}[xshift=10cm]
\draw (0,0) grid (2,2);
\draw[ultra thick,red] (0,0) to (0,1);
\draw[ultra thick,red] (0,1) to (1,1);
\draw[ultra thick,red] (1,1) to (1,2);
\draw[ultra thick,red] (1,2) to (2,2);
\end{scope}

\begin{scope}[xshift=12.5cm]
\draw (0,0) grid (2,2);
\draw[ultra thick,red] (0,0) to (0,1);
\draw[ultra thick,red] (0,1) to (1,2);
\draw[ultra thick,red] (1,2) to (2,2);
\end{scope}

\begin{scope}[xshift=15cm]
\draw (0,0) grid (2,2);
\draw[ultra thick,red] (0,0) to (0,2);
\draw[ultra thick,red] (0,2) to (2,2);
\end{scope}

\end{tikzpicture}

\end{center}
\caption{$(2,2)$ Delannoy paths}
\label{F:22}
\end{figure}

Let $L$ be a Delannoy path that ends at the 
lattice point $(p,q)\in \N$. 
We agree to represent $L$ as a word 
$L_1L_2\dots L_r$, where each $L_i$ $(i=1,\dots, r)$
is a pair of lattice points, say $L_i=((a,b),(c,d))$,
and $(c-a,d-b)\in \{N,E,D\}$. 
In this notation, we define the weight of $L_i$, the 
$i$-th step, as follows: 
\begin{align*}
weight(L_i) = 
\begin{cases}
1 & \text{ if } L_i =((a,b),(a+1,b)); \\
1 & \text{ if } L_i =((a,b),(a,b+1)); \\
a+b+1 & \text{ if } L_i= ((a,b),(a+1,b+1)).  
\end{cases}
\end{align*}
Finally, we define the weight of $L$, denoted by $\omega(L)$ 
as the product of the weights of its steps:
\begin{align}\label{A:weightofL}
\omega(L)  = weight(L_1) weight(L_2)\cdots weight(L_r).
\end{align}

\begin{Example}
Let $L$ denote the Delannoy path that is depicted in Figure~\ref{F:2D}.
In this case, the weight of $L$ is $\omega(L) = 3\cdot 6\cdot 8 = 144$.
\begin{figure}[h]
\begin{center}
\begin{tikzpicture}[scale=.65]
\begin{scope}
\node at (-1,3) {$L=$};
\draw (0,0) grid (4,6);
\draw[ultra thick,red] (0,0) to (0,1) node[anchor=north east] {1};
\draw[ultra thick,red] (0,1) to (0,2) node[anchor=north east] {1};
\draw[ultra thick,red] (0,2) to (1,3) node[anchor=north] {3};
\draw[ultra thick,red] (1,3) to (2,3) node[anchor=south east] {1};
\draw[ultra thick,red] (2,3) to (3,4) node[anchor=north] {6};
\draw[ultra thick,red] (3,4) to (4,5) node[anchor=north] {8};
\draw[ultra thick,red] (4,5) to (4,6) node[anchor=north west] {1};
\end{scope}
\end{tikzpicture}
\end{center}
\caption{A $(4,6)$ Delannoy path with weights on its steps.}
\label{F:2D}
\end{figure}
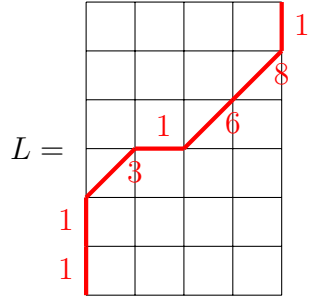

\end{Example}

\begin{Proposition}\label{P:anotherpaththeorem}
Let $p$ and $q$ be two nonnegative integers and 
let $\mc{D}(p,q)$ denote the corresponding set of Delannoy 
paths. 
In this case, we have 
$$
\alpha_{p,q} = \sum_{L\in \mc{D}(p,q)} \omega(L).
$$
\end{Proposition}

Although Proposition~\ref{P:anotherpaththeorem}
gives an expression for $\alpha_{p,q}$'s as a combinatorial sum,
it does not give a set of combinatorial objects whose
cardinality is given by $\alpha_{p,q}$. Our next result 
offers such an interpretation. 
\begin{Definition}
A $k$-diagonal step (in $\N^2$) is a diagonal step $L$ of the form 
$L=((a,b),(a+1,b+1))$, where $a,b\in \N$ and $k=a+b+1$.
(For an example, see Figure~\ref{F:4steps}.)
\end{Definition}
\begin{figure}[h]
\begin{center}
\begin{tikzpicture}[scale=.65]
\begin{scope}
\node at (-1,3) {$L=$};
\draw (0,0) grid (5,5);
\draw[ultra thick,red] (3,0) to (4,1);
\draw[ultra thick,red] (2,1) to (3,2);
\draw[ultra thick,red] (1,2) to (2,3);
\draw[ultra thick,red] (0,3) to (1,4);
\end{scope}
\end{tikzpicture}
\end{center}
\caption{The $4$-diagonal steps in $\N^2$.}
\label{F:4steps}
\end{figure}
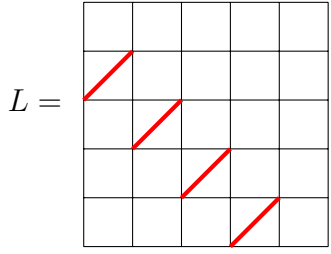

\vspace{1cm}

We proceed to define the ``weighted Delannoy paths.''
\begin{Definition}
By a labelled step we mean a pair $(K,m)$,
where $K\in \{N,E,D\}$ and $m$ is a positive integer
such that $m=1$ if $K=N$ or $K=E$. 
A weighted $(p,q)$ Delannoy path is a word of the form 
$W:=K_1\dots K_r$, where $K_i$'s $(i=1,\dots, r)$ 
are labeled steps $K_i=(L_i,m_i)$ such that 
\begin{itemize}
\item $L_1\dots L_r$ is a Delannoy path from $\mc{D}(p,q)$;
\item if $L_i$ ($1\leq i \leq r$) is a $k$-th 
diagonal step, then $1\leq m_i \leq  k-1$.
\end{itemize}
The set of all weighted $(p,q)$ Delannoy 
paths is denoted by $\mc{P}(p,q)$. 
\end{Definition}

\begin{Theorem}\label{T:anotherpaththeorem:intro}
There is a bijection between the set of weighted $(p,q)$ Delannoy 
paths and the set of signed $(p,q)$-involutions.
In particular, we have 
$$
\alpha_{p,q} = \sum_{W\in \mc{P}(p,q)} 1.
$$
\end{Theorem}

\vspace{.5cm}

As an application of our combinatorial results,
we compute the generating series
of the dimensions of $S(GL_p\times GL_q)$-orbits in $\ms{B}_n$.
To this end, we define $[m]:=1+\cdots + t^{m-1}$, where $m$ 
is a positive integer. 
Let $\mc{L}(\ms{B}_n)$ denote the set of all orbits of $S(GL_p\times GL_q)$ 
in $\ms{B}_n$. Finally, we define 
\begin{align}\label{A:E:intro}
E_{p,q}(t) := \sum_{Y\in \mc{L}(\mc{B}_n)} t^{\dim Y- \frac{3pq-q^2}{2}}.
\end{align}

({\em Note:} The (strange) shift on the exponents 
in (\ref{A:E:intro}) will be explained in Remark~\ref{R:explained}.)

\begin{Theorem}\label{T:lengthgf:intro}
Let $p$ and $q$ be two positive integers and let $E_{p,q}(t)$ be as in (\ref{A:E:intro}).
In this case, the polynomials $E_{p,q}(t)$'s satisfy the following 3-term recurrence:
$$
E_{p,q}(t) = E_{p-1,q}(t) + E_{p,q-1}(t) + ([p+q]-1) E_{p-1,q-1}(t).
$$
\end{Theorem}

We finish our introduction by briefly explaining  
the structure of our article. In Section~\ref{S:Preliminaries}
we introduce our basic notation regarding permutations
and involutions. In particular we introduce the 
length function of the Bruhat order on involutions, 
and give information about the Schubert varieties of grassmannians. 
The purpose of Section~\ref{S:Signed involutions} is to 
make an introduction to signed involutions and their generating 
series. The Section~\ref{S:Recurrence} is devoted to the
study of recurrence relations of $\gamma_{k,p,q}$'s and
$\alpha_{p,q}$'s; the proofs of Theorem~\ref{T:first main result:intro} 
and its Corollary~\ref{C:first main result:intro} are presented 
in that section. The generating series of $\alpha_{p,q}$'s are 
studied in Section~\ref{S:Generating}, where we prove 
Theorem\ref{T:gs:intro} and its Corollary~\ref{C:genfun:intro}.
Our main combinatorial interpretation for the numbers
$\alpha_{p,q}$'s are recorded in Section~\ref{S:Interpretation}.
The final Section~\ref{S:Analogs} is devoted to the study 
of various generating functions. We give a proof of 
Theorem~\ref{T:lengthgf:intro}, and, in addition, 
we introduce and study two more generating polynomials,
which are defined as follows:
$$
D_{p,q}(t)= \frac{1}{t} \sum_{L\in \mc{D}(p,q)} t^{\omega(L)} \qquad
\text{and}\qquad A_{p,q}(t) = \sum_{k=0}^q \gamma_{k,p,q} t^k.
$$
We show that both of these polynomials satisfy 
recurrence relations that are similar to that of the $E_{p,q}(t)$'s
as in Theorem~\ref{T:lengthgf:intro}.

\section{Preliminaries}\label{S:Preliminaries}

For two subgroups $H$ and $K$ from a group $G$, there is a 
bijection between the set of $H$-orbits in $G/K$ and the set of $K$-orbits in $G/H$. 
This is easily seen by passing to the $(H,K)$-double cosets in $G$. 
Let $\theta_i$ ($i=1,2,3$) be one of the 
involutions mentioned in the introduction 
and let $K_i:=SL_{n}^{\theta_i}$ denote 
the corresponding symmetric subgroup in $SL_{n}$. 
As we mentioned it earlier, 
the $K_i$-orbits in $\ms{B}_n$ are parametrized 
by the corresponding sets of involutions in $S_n$. 
Equivalently, these sets of involutions parametrize 
the $B_n$-orbits in the quotients $SL_{n}/K_i$.
Note that the set of all 
Borel orbit closures in an affine variety 
form a graded poset with respect to inclusion.
We will refer to such a poset as a {\em Bruhat poset}. 
The Bruhat posets corresponding to $I_n$ and $FI_n$
are first considered by 
Richardson and Springer in their seminal paper~\cite{RS}. 
In the next few paragraphs we will 
briefly review some of the well-known results on these 
objects and set up our basic notation on involutions.

We write the elements of the symmetric group $S_n$ 
in cycle notation using parentheses, as well as in one-line notation using brackets. 
We omit brackets in one-line notation if there is no danger of confusion.   
For example, $w=4213=[4,2,1,3]=(1,4,3)$ is the permutation that maps 
$1$ to $4$, $2$ to $2$, $3$ to $1$, and $4$ to $3$.

An involution is an element of $S_n$ of order 
$\leq 2$ and the set of involutions in $S_n$ is denoted by $I_n$. 
The Bruhat order on $I_n$ has a minimal element 
$\alpha_n := \text{id}$, and a maximal element 
$w_0 = [n, n-1, \dots, 2, 1]$, which is 
the longest permutation in $S_n$.

Let $\pi \in I_n$ be an involution. 
The standard way of writing $\pi$ 
is as a product of 2-cycles. 
Since we often need the data of fixed points (1-cycles) of $\pi$, 
we are always going to include them in our notation. 
Thus, our {\em standard form} for $\pi$ is 
$$
\pi = (a_1, b_1)(a_2, b_2) \dots (a_k, b_k) c_1\dots c_{n-2k},
$$ 
where $a_i < b_i$ for all $1 \leq i \leq k$, 
$a_1 < a_2 < \dots < a_k$, and $c_1<\dots < c_{n-2k}$.

The Bruhat-Chevalley order on $S_n$ is a ranked poset 
and its grading is given by $\ell : S_n \rightarrow \Z$,
$\ell(\pi) = \text{the number of inversions in } \pi$. 
The Bruhat order on $I_n$ is also a ranked poset; 
for $\pi = (a_1, b_1) \dots (a_k, b_k)c_1\dots c_{n-2k} \in I_n$, 
the length $\mathbb{L}(\pi)$ is defined by
\begin{align}\label{A:it is observed}
\mathbb{L}(\pi) := \frac{\ell(\pi)+k}{2},
\end{align}
where $\ell(\pi)$ is the length of $\pi$ in $S_n$, and $k$ is the number 
$2$-cycles that appear in the standard form of $\pi$.

The papers~\cite{CCT,Incitti} give precise descriptions 
of the covering relations of Bruhat orders on the (fixed point free)
involutions in $S_n$. Although it is presented in a slightly different 
terminology, by using $(p,q)$-clans in~\cite{Wyser16}, Wyser 
described the Bruhat order on signed involutions. 

\vspace{.5cm}

There is another ``geometric'' partial ordering on involutions, 
namely the weak order, that is extremely useful for studying 
Bruhat ordering. 
In particular, the weak order is a ranked poset and its length 
function agrees with that of the Bruhat order. In general, 
a convenient way of defining the weak order is via the so called 
"Richardson-Springer monoid" action. Since our goal in this 
paper is {\em not} studying the poset structure, and since the 
descriptions of both Bruhat and weak orders are lengthy  
we skip their definitions but mention a relevant fact. 
(See Figure~\ref{F:Weak22}, 
where we illustrate the weak order on $I_{2,2}^\pm$.)

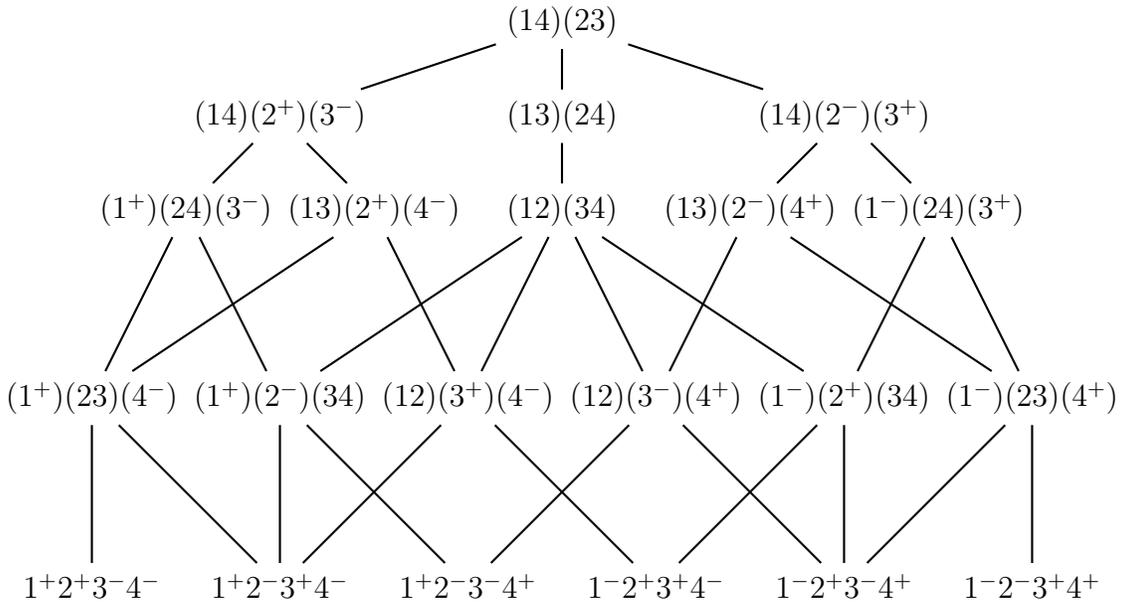
\begin{figure}[htp]
\begin{center}

\begin{tikzpicture}[scale=.25]

\node at (-25,0) (a1) {$1^+ 2^+ 3^- 4^-$};
\node at (-15,0) (a2) {$1^+ 2^- 3^+ 4^-$};
\node at (-5,0) (a3) {$1^+ 2^- 3^- 4^+$};
\node at (5,0) (a4) {$1^- 2^+ 3^+ 4^-$};
\node at (15,0) (a5) {$1^- 2^+ 3^- 4^+$};
\node at (25,0) (a6) {$1^- 2^- 3^+ 4^+$};

\node at (-25,10) (b1) {$(1^{+})(23)(4^{-})$};
\node at (-15,10) (b2) {$(1^{+})(2^{-})(34)$};
\node at (-5,10) (b3) {$(12)(3^+)(4^-)$};
\node at (5,10) (b4) {$(12)(3^-)(4^+)$};
\node at (15,10) (b5) {$(1^-)(2^+)(34)$};
\node at (25,10) (b6) {$(1^{-})(23)(4^{+})$};

\node at (-20,20) (c1) {$(1^{+})(24)(3^{-})$};
\node at (-10,20) (c2) {$(13)(2^{+})(4^{-})$};
\node at (0,20) (c3) {$(12)(34)$};
\node at (10,20) (c4) {$(13)(2^{-})(4^{+})$};
\node at (20,20) (c5) {$(1^{-})(24)(3^{+})$};

\node at (-15,25) (d1) {$(14)(2^{+})(3^{-})$};
\node at  (0,25) (d2) {$(13)(24)$};
\node at  (15,25) (d3) {$(14)(2^{-})(3^{+})$};

\node at (0,30) (e) {$(14) (23)$};

\draw[-,  thick] (a1) to (b1);
\draw[-,  thick] (a2) to (b1);
\draw[-,  thick] (a2) to (b2);
\draw[-,  thick] (a2) to (b3);
\draw[-,  thick] (a3) to (b2);
\draw[-,  thick] (a3) to (b4);
\draw[-,  thick] (a4) to (b3);
\draw[-,  thick] (a4) to (b5);
\draw[-,  thick] (a5) to (b4);
\draw[-,  thick] (a5) to (b5);
\draw[-,  thick] (a5) to (b6);
\draw[-,  thick] (a6) to (b6);

\draw[-, thick] (b1) to (c1);
\draw[-,  thick] (b1) to (c2);
\draw[-, thick] (b2) to (c1);
\draw[-, thick] (b2) to (c3);
\draw[-,  thick] (b3) to (c2);
\draw[-, thick] (b3) to (c3);
\draw[-, thick] (b4) to (c3);
\draw[-, thick] (b4) to (c4);
\draw[-, thick] (b5) to (c3);
\draw[-, thick] (b5) to (c5);
\draw[-, thick] (b6) to (c4);
\draw[-, thick] (b6) to (c5);

\draw[-, thick] (c1) to (d1);
\draw[-, thick] (c2) to (d1);
\draw[-, thick] (c3) to (d2);
\draw[-, thick] (c4) to (d3);
\draw[-, thick] (c5) to (d3);

\draw[-, thick] (d1) to (e);
\draw[-, thick] (d2) to (e);
\draw[-, thick] (d3) to (e);
\end{tikzpicture}

\caption{Weak order on $I_{2,2}^{\pm}$}\label{F:weak order I_4}
\label{F:Weak22}
\end{center}
\end{figure}

We will occasionally mention the ``support of a signed involution'' 
to mean the underlying involution without reference to its signs.

\begin{Remark}\label{R:it is observed}
It is observed in~\cite{CMW} that the length function of the weak 
order on signed involutions $I_{p,q}^\pm$ agrees with that of the 
weak order on $I_n$. In other words, the length of $\pi\in I_{p,q}^\pm$ is equal to 
$\mathbb{L}(\pi)$, where $\pi$ is identified with its supporting involution. 
From now on, by abusing notation we are going to use $\mathbb{L}(\cdot)$ 
for denoting the length function of Bruhat as well as the weak order on 
$I_{p,q}^\pm$.
\end{Remark}

Our final remark on the length function $\mathbb{L}(\cdot)$ is that if 
$\pi \in I_{p,q}^\pm$ is the signed involution corresponding to the Borel orbit 
$\mc{O}$, then the dimension of $\mc{O}$ is equal to $\mathbb{L}(\pi) + c$, 
where $c$ is the dimension of (any) closed Borel orbit in 
$SL_n/S(GL_p \times GL_q)$. Thus, studying $\mathbb{L}(\pi)$
is equivalent to studying dimensions of $S(GL_p\times GL_q)$-orbits.

\begin{Remark}\label{R:GrassmanPaths}
Let $V$ denote the complex $n$-dimensional vector space 
and let $\ms{G}_{p,n}$ denote the 
Grassmann variety of $p$-dimensional subspaces in $V$.
Let $P_{p,n}$ denote the subgroup of $SL_n$  
consisting of matrices of the form 
$\begin{bmatrix} A & B \\ \mathbf{0} & C \end{bmatrix}$ where $A$ and $C$ are
respectively $p\times p$ and $(n-p) \times (n-p)$ block-matrices with $\det A \cdot \det C =1$. 
In this case, the quotient space $SL_n / P_{p,n}$ is canonically 
isomorphic to $\ms{G}_{p,n}$, hence it parametrizes the 
$p$-dimensional subspaces of $\C^n$. Note that $SL_n$ acts on $\ms{G}_{p,n}$
by the left multiplication. A Schubert variety in $\ms{G}_{p,n}$ is defined 
as the Zariski closure of a $B_n$-orbit in $SL_n / P_{p,n}$. 
The following facts are well-known (see
the first section of Brion's lecture notes~\cite{Brion-Lectures}):
\begin{enumerate}
\item The set of $B_n$-orbits in $\ms{G}_{p,n}$ 
(hence, the set of $P_{p,n}$-orbits in $\ms{B}_n$) 
are in bijection with the lattice paths in $\N^2$
that starts at the origin and ends at the point $(n-p,p)$, moving 
with unit steps $E:=(1,0)$ and $N:=(0,1)$. We will refer to these 
lattice paths by $(p,q)$ Grassmann paths. Here, $q=n-p$. 
See Figure~\ref{F:grassmannian} for an example of a $(3,7)$
Grassmann path. 
\item Let $I$ be a $(p,q)$ Grassmann path and let $C_I$ 
denote the corresponding $B_n$-orbit. 
Then $\dim C_I$ is the total number of squares that are weakly 
above $I$ and contained inside the $p$-by-$q$ rectangular 
region. For example, for the Grassmann path in Figure~\ref{F:grassmannian}
the dimension is equal to $8$.
\item 
Let $I$ and $J$ be two $(p,q)$ Grassmann paths and 
let $C_I$ and $C_J$ denote,
respectively, the $B_n$-orbits in $\ms{G}_{p,n}$. 
In this case, $I$ is weakly below $J$ if and only if 
$C_J\subseteq \overline{C_I}$. 
\item The $B_n$-orbits that are contained in 
$Y:=\overline{C_I}$ give a cellular decomposition for $Y$.
Therefore, the topology of $Y$ is completely determined by
the $(p,q)$ Grassmann paths that stay weakly above $I$.
\end{enumerate}

\begin{figure}[htp]
\begin{center}
\begin{tikzpicture}[scale=.7]
\begin{scope}
\draw[thick, dotted] (0,0) grid (7,3);
\node at (7.5,3.5) {$(7,3)$};
\node at (-.5,-.5) {$(0,0)$};
\draw[ultra thick,red] (0,0) to (0,1);
\draw[ultra thick,red] (0,1) to (2,1);
\draw[ultra thick,red] (2,1) to (2,2);
\draw[ultra thick,red] (2,2) to (6,2);
\draw[ultra thick,red] (6,2) to (6,3);
\draw[ultra thick,red] (6,3) to (7,3);
\end{scope}
\end{tikzpicture}
\end{center}
\caption{A lattice path corresponding to a Schubert variety in $\ms{G}_{3,10}$.}
\label{F:grassmannian}
\end{figure}
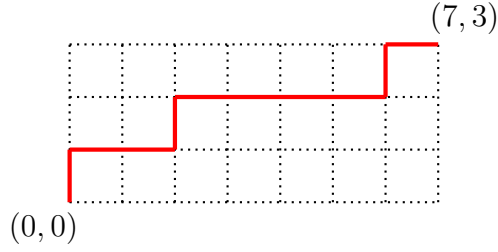

\end{Remark}

\section{Signed Involutions with $k$ 2-cycles}\label{S:Signed involutions}

Let $n$ and $k$ be two integers such that 
$0 \leq k \leq \lfloor \frac{n}{2} \rfloor$ and let 
$\pi = (i_1 j_1) \dots (i_k j_k) l_1 \dots l_{n-2k}$ 
be an involution with fixed points $l_1 < \dots < l_{n-2k}$.
Let us assume that $p$ and $q$ are two 
numbers such that $0\leq q \leq p$ and $p+q=n$.

We denote by $I^\pm_{k,p,q}$ the set of all signed 
involutions $\pi$ from $I_{p,q}^\pm$ 
such that $\pi$ has $n-2k$ fixed points, $p - q \leq n - 2k$, 
and there are $p- q$ more $+$'s than $-$'s. 
We denote the cardinality of $I^\pm_{k,p,q}$ by $\gamma_{k,p,q}$.

\begin{Remark}
Since $q+p-2k= n-2k \geq p-q$, it holds true that $0 \leq k \leq q$.
\end{Remark}

Now, if $\pi\in I^\pm_{k,p,q}$, then its support is one of the 
${ n \choose 2k } \frac{(2k)!}{2^k k!}$ involutions
$\pi \in I_n$ with $n-2k$ fixed points. 
This is easy to see but let us justify it for completeness:
Once the entries to appear in the transpositions of $\pi$ are chosen, 
the fixed points of $\pi$ are uniquely determined, so, 
the question is equivalent to choosing $k$ transposition
from $[n]$ and ordering them to give $\pi$. 
Now, it is not difficult to see that this number is given by 
${ n \choose 2k } \frac{(2k)!}{2^k k!}$.

Next, we look into ways to place $a$ +'s and $b$ -'s on the string 
$l_1 \dots l_{n-2k}$ so that there are exactly 
$p-q=a-b$ +'s more than -'s. 
Clearly, this number is equivalent to ${ n- 2k \choose a }$. 
Since $a+b = n-2k= q+p-2k$ and $a-b = p-q$, we have $a = p-k$. 
Therefore, we have 
\begin{align}\label{A:first formula}
\gamma_{k,p,q} = { q+p \choose 2k } \frac{ (2k)! }{ 2^k k! } { q+ p - 2k  \choose q-k }.
\end{align}
We express (\ref{A:first formula}) more symmetrically as follows:
\begin{align}\label{A:gamma kqp}
\gamma_{k,p,q} = \frac{(q+p)! }{ (q-k)! (p-k)! } \frac{ 1 }{ 2^k k! }.
\end{align}

Observe that the formula (\ref{A:gamma kqp}) 
is defined independently of the inequality $q<p$. 
From now on, for our combinatorial purposes, 
we skip mentioning this comparison between $p$ and $q$ and 
use the equality $\gamma_{k,p,q}=\gamma_{k,q,p}$ 
whenever it is needed. 
Also, we record the following obvious recurrences for 
a future reference:
\begin{align}\label{A:Future}
\gamma_{k,p,q} = \frac{(p-k+1)(q-k+1)}{2k} \gamma_{k-1,p,q}, \;\;
\gamma_{k,p,q} = \frac{p+q}{p-k} \gamma_{k,p-1,q}, \;\;
\gamma_{k,p,q} = \frac{p+q}{q-k} \gamma_{k,p,q-1}.
\end{align}

\vspace{.5cm}

Let $c_n$ denote the number of involutions in $S_n$. 
In other words, $c_n= | I_n |$. The exponential generating 
function of $c_n$'s is given by $e^{t + \frac{t^2}{2}}$. 
Let us define the polynomials $K_n(x)$ by the equality 
$$
\sum_{n\geq 0} K_n(x) \frac{t^n}{n!}  = e^{xt+\frac{t^2}{2}}.
$$
It is well known that $K_n(x) = \sum_{\pi \in I_n} x^{a_1(\pi)}$, 
where $a_1(\pi)$ denotes the number of 1-cycles (fixed points)
of $\pi$. See~\cite[Exercise 5.19]{Stanley}.
It easily follows that if $c_{n,r}$ denotes the number of elements of 
$I_n$ with exactly $r$ 1-cycles, then 
\begin{align}\label{A:2-cycles}
\sum_{n,r\geq 0} c_{n,r} \frac{t^n}{n!} x^r = e^{x t+\frac{t^2}{2}}.
\end{align}

\begin{Remark}\label{R:note}
The numbers $c_{n,r}$ appear in our context rather naturally. 
Suppose we have $q=k \leq p$. Then $\gamma_{k,p,k}=\gamma_{k,k,p}$ is the 
number of signed involutions on $[p+k]$ such that there are $p-k$ +'s more than -'s 
on the fixed points. It is not difficult to see in this case 
that the number of $-$'s is 0. 
Therefore, $\gamma_{k,p,k}$ is the number of involutions on $[p+k]$ with $p-k$ 1-cycles and whose fixed 
points have + signs only. In other words, $\gamma_{k,p,k}= c_{p+k,p-k}$. 
The polynomial $K_n(x)$ is the sum of the entries of the $n$-th row of Table~\ref{T:H}.
\begin{table}[htp]
$$
\begin{matrix}
c_{0,0}x^0 & 0 & 0 & 0 & 0 &\cdots \\
c_{1,0}x^0 & c_{1,1}x^1 & 0& 0 & 0 & \cdots \\
c_{2,0}x^0 & c_{2,1}x^1 & c_{2,2} x^2 & 0 & 0 & \cdots \\
c_{3,0}x^0 & c_{3,1}x^1 & c_{3,2} x^2 & c_{3,3} x^3  & 0 & \cdots \\
\end{matrix}
$$
\caption{Dissection of $K_n(x)$'s.}
\label{T:H}
\end{table}

\end{Remark}

\section{Recurrences}\label{S:Recurrence}

In this section, we will show that $\gamma_{k,p,q}$'s obey a 3-term 
recurrence and we will exploit the consequences of this fact. 

\begin{Theorem}\label{T:rec}
Let $p$ and $q$ be two positive integers. 
Then the following recurrence relation and its initial condition hold true:
\begin{equation}\label{recurrence}
\gamma_{k,p,q} =  \gamma_{k,p-1,q} +  
\gamma_{k,p,q-1} + (q+p-1) \gamma_{k-1,p-1,q-1}\; \; \text{and} \;\; 
\gamma_{0,p,q}= {p+q \choose q}.
\end{equation}
Furthermore, the exact formula for $\gamma_{k,p,q}$ is given by 
\begin{align}\label{A:gamma kqp 1}
\gamma_{k,p,q} = \frac{(q+p)! }{ (q-k)! (p-k)! } \frac{ 1 }{ 2^k k! }.
\end{align}
\end{Theorem}
\begin{proof}

We already proved eqn. (\ref{A:gamma kqp 1}) in the previous section.
The proof of the initial condition is easy, so, we omit it.
Now, we are going to construct our proof of the recurrence 
by analyzing what happens to an involution $\pi \in I_{k,p,q}^\pm$ 
when we remove its largest entry $n$. Clearly, $n$ appears in $\pi$ 
either as a fixed point, or in one of the 2-cycles. 
Thus, we partition $I_{k,p,q}^\pm$ into $n+1$ disjoint subsets; 
$$
I_{k,p,q}^\pm = I_{k,p,q}^\pm (+) \cup I_{k,p,q}^\pm (-) \cup \bigcup_{i=1}^{n-1} I_{k,p,q}^\pm (i),
$$ 
where
\begin{enumerate}
\item $I_{k,p,q}^\pm (+):=\{ \pi \in I_{k,p,q}^\pm :\ \text{ n is a fixed point with a + sign}\}$,
\item $I_{k,p,q}^\pm (-):=\{ \pi \in I_{k,p,q}^\pm :\ \text{ n is a fixed point with a - sign}\}$,
\item $I_{k,p,q}^\pm (i):=\{ \pi \in I_{k,p,q}^\pm :\ \text{ n appears in the 2-cycle $(i,n)$}\}$ for $i=1,\dots, n-1$.
\end{enumerate}

First, we assume that $\pi \in I_{k,p,q}^\pm (+)$, so 
$$
\pi = (i_1\; j_1) \dots (i_k \; j_k) l_1 \dots l_{n-2k-1} n^+.
$$ 
It follows that by removing $n$ we reduce the total number of 
$+$ signs by $1$. Note that this makes sense because 
$p-q$ is fixed. Thus, the number of such signed $(p, q)$-involutions 
on $[n]$ is counted by $\gamma_{k,p-1,q}$.  
By using a similar argument for the case $\pi \in I_{k,p,q}^\pm (-)$, 
we conclude that there are $\gamma_{k,p,q-1}$ such signed involutions on $[n]$.

Next, and finally, we consider the case where 
$n$ appears in a 2-cycle $(i, j)$ of $\pi$. Then $j=n$ and therefore   
$\pi \in I_{k,p,q}^\pm (i)$. It is obvious that there are $(n-1)$ 
possibilities for $i$ from $n$ numbers. Removing this 2-cycle 
from $\pi$ leaves us with $(n-2)$ elements and $(k-1)$ 2-cycles 
but it does not change the signs on the fixed points. 
Of course, once this 2-cycle is removed, 
we decrease the numbers that are greater than $i_r$ by 1 
so that we have valid signed involution whose support lies in $I_{n-2}$. 
In particular, since the difference of $+$ and $-$ 
signs is preserved, we see that the number of such signed 
$(p, q)$-involutions on $[n]$ is given by 
$(n-1)\gamma_{k-1,p-1,q-1}$. 

Notice that in each of these cases we get an injective 
map into a set of signed involutions of a smaller size. 
Indeed, by removing $n$, in the cases 1. and 2. 
we get injections into $I_{k,p-1,q}^\pm$ and $I_{k,p,q-1}^\pm$, respectively.
In the case of 3. we get an injection into $I_{k-1,p-1,q-1}^\pm$. 
Conversely, if $\pi'$ is a signed involution from $I_{k,p-1,q}^\pm$ 
(or, from $I_{k,p,q-1}$), then we append $n^+$
(resp. $n^-$) to get an element $\psi_k(+)(\pi') \in I_{k,p,q}^\pm$ 
(resp. $\psi_k(-)(\pi') \in I_{k,p,q}^\pm$). If 
$\pi' \in I_{k-1,p-1,q-1}^\pm$, then we pick a number, 
say $i \in [n-1]$ in $n-1$ different ways; we add 1 to every number 
$j$ such that $i < j$ and $j$ appears in the standard form of $\pi'$; 
and finally we insert the 2-cycle $(i,n)$ into $\pi'$. Let us denote 
the resulting map by $\psi_k(i)(\pi') \in I_{k,p,q}^\pm$
Obviously, the maps $\psi(\pm)$ and $\psi(i)$ 
are well defined inverses to the procedures that are described in the 
previous paragraph. 
Thus, it is now clear that we have built a bijection between $I_{k,p,q}^\pm$ and 
the disjoin union 
$I_{k,p-1,q}^\pm \cup I_{k,p,q-1}^\pm \cup \bigsqcup_{i=1}^{n-1} I_{k-1,p-1,q-1}^\pm$, 
proving our claim. ({\em Note:} Here, we use square-union to indicate 
that it is a disjoint union of $n-1$ copies of the same set.)

\end{proof}

We list two important corollaries.

\begin{Corollary}\label{P:note}
The number of involutions $\pi \in I_n$ with exactly $r$ 1-cycles, 
$c_{n,r}$, is a special case of $\gamma_{k,p,q}$'s.
\end{Corollary}
\begin{proof}
We already know from Remark~\ref{R:note} that $\gamma_{k,p,k}= c_{p+k,p-k}$. 
The result follows from the fact that the equations 
$n= p+k$, $r=p-k$ have a unique solution.
\end{proof}

\begin{Corollary}\label{C:First recurrence}
If we put $\alpha_{p,0}=\alpha_{0,q}=1$ for all nonnegative integers $p$ and $q$, 
then the numbers $\alpha_{p,q}$ satisfy the following recurrence relation
\begin{equation}\label{A:recurrence}
\alpha_{p,q} =  \alpha_{p-1,q} + \alpha_{p,q-1} +  (p+q-1) \alpha_{p-1,q-1},\qquad p,q\geq 1.
\end{equation}
\end{Corollary}

\begin{proof}
Taking the sum of both sides of the equation (\ref{recurrence}) over $k$, 
where $1 \leq k \leq q-1$, gives us
\begin{align*}
\alpha_{p,q} -  \alpha_{p-1,q} -  \alpha_{p,q-1} - (p+q-1) \alpha_{p-1,q-1} 
&= \gamma_{0,p,q} -  \gamma_{0,p-1,q} -  \gamma_{0,p,q-1} \\
&+ \gamma_{q,p,q} -  \gamma_{q,p,q-1} - (p+q-1) \gamma_{q-1,p-1,q-1} \\
&= 0.
\end{align*}
By using the recurrence relation in (\ref{recurrence}) once more, 
we see that the latter becomes $0$.   
\end{proof}

\section{Generating functions}\label{S:Generating}

One of the many options for a bivariate generating function 
for $\alpha_{p,q}$'s is the following power series:
\begin{align}\label{A:one of many}
v(x,y) := \sum_{p,q \geq 0 } \alpha_{p,q} x^q \frac{y^p}{p!}.
\end{align}
We will find a closed formula for (\ref{A:one of many}) 
by using the recurrence relation (\ref{A:recurrence}).
Multiplying both sides of the recurrence relation by $\frac{x^q y^p}{p!}$ 
and taking the sum over all $p,q \geq 1$ give us 
\begin{equation}\label{E:first}
\sum_{p,q \geq 1} \frac{\alpha_{p,q}}{p!} x^q y^p = \sum_{p,q \geq 1} \frac{\alpha_{p-1,q}}{p!} x^q y^p+ \sum_{p,q \geq 1} \frac{\alpha_{p,q-1}}{p!} x^q y^q+   \sum_{p,q \geq 1} (p+q-1) \frac{\alpha_{p-1,q-1}}{p!} x^q y^p.
\end{equation}
Since
\begin{align*}
v(x, y) = \sum_{p,q \geq 0} \frac{\alpha_{p,q}}{p!} x^q y^p 
&= \alpha_{0,0} + \alpha_{0,1} x + \dots + \alpha_{0,q} x^q +\dots \\
&+ \frac{\alpha_{1,0}}{1!} y + \dots + \frac{\alpha_{p,0}}{p!} y^p + \dots\\
&+ \frac{\alpha_{1,1}}{1!} xy + \dots + \frac{\alpha_{p,1}}{p!} xy^p +\dots\\
&+ \frac{\alpha_{1,2}}{1!} x^2 y + \dots + \frac{\alpha_{p,2}}{p!} x^2 y^p +\dots
\end{align*}
the equation (\ref{E:first}) combined with the initial conditions $\alpha_{p,0}=\alpha_{0,q}=1$ gives 
\begin{align*}
v(x,y)- \frac{1}{1-x} - e^y +1 &= \int \sum_{p\geq 1,q \geq 0} \frac{\alpha_{p-1,q}}{(p-1)!} x^{q} y^{p-1}  dy - e^y 
+ x \bigg(\sum_{p,q \geq 0} \frac{\alpha_{p,q}}{p!} x^q y^p - \frac{1}{1-x} \bigg) \\
&+ \sum_{p,q \geq 1} p\frac{\alpha_{p-1,q-1}}{p!} x^q y^p 
+ \sum_{p,q \geq 1} q\frac{\alpha_{p-1,q-1}}{p!} x^q y^p -   \sum_{p,q \geq 1} \frac{\alpha_{p-1,q-1}}{p!} x^q y^p \\
&= \int \sum_{p\geq 1,q \geq 0} \frac{\alpha_{p-1,q}}{(p-1)!} x^{q} y^{p-1}  dy - e^y 
+ x \bigg(\sum_{p,q \geq 0} \frac{\alpha_{p,q}}{p!} x^q y^p - \frac{1}{1-x} \bigg) \\
&+ xy \sum_{p,q \geq 0} \frac{\alpha_{p-1,q-1}}{(p-1)!} x^{q-1} y^{q-1} 
+\int  \sum_{p,q \geq 1} \frac{q\alpha_{p-1,q-1}}{(p-1)!} x^{q-1} y^{p-1} dy\\
 &-  x \int \sum_{p,q \geq 1} \frac{\alpha_{p-1,q-1}}{(p-1)!} x^{q-1} y^{p} dy.
\end{align*}
Thus, we have 
\begin{align*}
v(x, y)  - \frac{1}{1-x} - e^y +1 &= \int v(x,y) dy - e^y + xv(x,y) - \frac{x}{1-x} + xyv(x,y)\\
 &+ x \int \bigg( \frac{\partial}{\partial x} (x v(x,y))\bigg) dy - x \int v(x,y) dy,
\end{align*}
or equivalently,
\begin{align*}
(1-x- xy) v(x, y) &= (1-x)\int v(x,y) dy  + x \int \bigg( \frac{\partial}{\partial x} (x v(x,y))\bigg) dy .
\end{align*} 

Taking the integral of both sides with respect to $y$ yields the following PDE
\begin{align*}
-xv(x,y) +(1-x-xy) \frac{\partial v(x,y)}{\partial y} 
= (1-x) v(x,y) + x \bigg( v(x,y) + x \frac{\partial v(x,y)}{\partial x} \bigg),
\end{align*}
or, equivalently, 
\begin{align}\label{A:PDE}
(-x^2) \frac{\partial v(x,y)}{\partial x} + (1-x-xy)\frac{\partial v(x,y)}{\partial y}=  (1+x)v(x,y),
\end{align}
with the initial conditions $v(0,y) = e^y$ and $v(x,0) = \frac{1}{1-x}$.

Solutions of such PDE's are easily obtained by applying the method of ``characteristic curves.''
Our characteristic curves are $x(r,s), y(r,s)$, and $v(r,s)$. Their tangents are equal to 
\begin{equation}\label{charcurve}
\frac{\partial x }{\partial s } = -x^2 \qquad \frac{\partial y }{\partial s } 
= 1- x- xy, \qquad \frac{\partial v }{\partial s } = (1+x)v,
\end{equation}
with the initial conditions $$x(r,0) = r,\;\; y(r,0) =0,\;\; \text{and}\;\; v(r,0) = \frac{1}{1-r}.$$

From the first equation given in (\ref{charcurve}) 
and its initial condition below it,  we compute that  
\begin{align}\label{A:y is simple}
x(r,s)= \frac{r}{rs+1}.
\end{align}
Plugging this into the second equation gives us 
$\frac{\partial y }{\partial s } =  1- \frac{r}{rs+1}(1+y)$, 
which is a first order linear ODE. Its solution is given by 
\begin{align}\label{A:x is complicated}
y(r,s)=\frac{rs^2 - 2rs + 2s }{2(rs+1)}.
\end{align}
Finally, from the last equation in (\ref{charcurve}) 
together with its initial condition we conclude that 
$$
v(r,s) =\frac{ e^{s} (rs+1)}{1-r}.
$$
In summary we outlined the proof of our next result.
\begin{Theorem}\label{T:GF2}
Let $v(x,y)$ denote the function that is represented by the series 
$\sum_{p,q\geq 0} \alpha_{p,q} x^q \frac{y^p}{p!}$
around the origin. If $r$ and $s$ are the variables related to $x$ and $y$ as in equations 
(\ref{A:x is complicated}) and (\ref{A:y is simple}), then around $(r,s)=(0,0)$ we have 
\begin{align}\label{A:simple}
v(r,s) =\frac{ e^{s} (rs+1)}{1-r}.
\end{align}
\end{Theorem}

\begin{Corollary}\label{C:genfun}
The bivariate generating series for $\alpha_{p,q}$'s is given by 
\begin{align}\label{A:substitute:intro}
\sum_{p,q\geq 0} \alpha_{p,q} x^q \frac{y^p}{p!}
&=\frac{e^{\frac{1-x-\sqrt{x^2-2xy-2x+1}}{x}} (x-\sqrt{x^2-2xy-2x+1})}{-x^2+2xy+2x-1+x\sqrt{x^2-2xy-2x+1}}.
\end{align}
\end{Corollary}

\begin{proof}
The general solution $S(x,y)$ of (\ref{A:PDE}) is given as follows:
\begin{align}\label{A:PdeSolt}
S(x, y)= \frac{ e^{1/x} F(\frac{2xy+2x-1}{2x^2})}{x}
\end{align}
where $F(z)$ is some function in one-variable. 
(This can be verified by substituting (\ref{A:PdeSolt}) into (\ref{A:PDE}).)

We want to choose $F(z)$ in such a way that $S(x,y)=v(x,y)$
holds true. To do so, first, consider the case where $y=0$, 
then by definition $v(x,0) = \frac{1}{1-x}$.
Thus, $F(z)$ satisfies the following equation  
\begin{align}\label{A:Gsolt}
\frac{ e^{1/x} F(\frac{2x-1}{2x^2})}{x} = \frac{1}{1-x} \;\;\; \text{or} \;\;\;
F \biggr(\frac{2x-1}{2x^2} \biggr) = \frac{xe^{-1/x}}{1-x}.
\end{align}
Now we take the inverse of the transformation $z= \frac{2x-1}{2x^2}$ 
in (\ref{A:Gsolt}) to have $x=\frac{1- \sqrt{1-2z}}{2z}$ and thus,
\begin{align}
F(z) = \frac{e^{-1-\sqrt{1-2z}} (1-\sqrt{1-2z})}{2z-1+\sqrt{1-2z}}.
\end{align}
Next, we substitute this into the general solution (\ref{A:PdeSolt}) 
and conclude that 
\begin{align*}
v(x,y) &= \frac{{e^{\frac{1-x}{x}- \frac{\sqrt{x^2-2xy-2x+1}}x}}\biggr(\frac{x-\sqrt{x^2-2xy-2x+1}}{x}\biggr)} {\frac{-x^2+2xy+2x-1+x\sqrt{x^2-2xy-2x+1}}{x}} \notag \\
&=\frac{e^{\frac{1-x-\sqrt{x^2-2xy-2x+1}}{x}} (x-\sqrt{x^2-2xy-2x+1})}{-x^2+2xy+2x-1+x\sqrt{x^2-2xy-2x+1}}
\end{align*}
This finishes the proof.
\end{proof}

\section{Combinatorial interpretations}\label{S:Interpretation}

In this section, we will give the proofs of Proposition~\ref{P:anotherpaththeorem}
and Theorem~\ref{T:anotherpaththeorem:intro}. 
Recall that the set of all $(p,q)$ Delannoy paths is denoted by 
$\mc{D}(p,q)$ and the weight of an element $L\in \mc{D}(p,q)$ 
is the product of the weights of the steps in $L$. 
Recall also that the weight of a step $L_i=((a_i,b_i),(a_{i+1},b_{i+1}))$ in 
$L$ is defined $a_{i+1}+b_{i+1}-1$ if $L_i$ is a diagonal step; it is 
1, otherwise.

\begin{Proposition} 
If $p$ and $q$ are two nonnegative integers, then 
\begin{align}\label{L:First combinatorial}
\alpha_{p,q} = \sum_{L \in \mc{D}(p,q)} \omega(L).
\end{align}
\end{Proposition}

\begin{proof}

The following special cases are easy to verify:
\begin{align*}
\alpha_{p,0} = \sum_{L \in \mc{D}(p,0)} \omega(L),\ 
\alpha_{0,q} = \sum_{L \in \mc{D}(0,q)} \omega(L),\ \text { and }
\alpha_{1,1} = \sum_{L \in \mc{D}(1,1)} \omega(L).
\end{align*}
Thus, to prove (\ref{L:First combinatorial}),
it suffices to prove that the right hand side of it 
satisfies the recurrence in Corollary~\ref{C:First recurrence}.

To this end, we partition $\mc{D}(p,q)$ into three disjoint subsets,
denoted by $\mc{D}_N(p,q)$, $\mc{D}_E(p,q)$, and $\mc{D}_D(p,q)$, 
where $\mc{D}_A(p,q)$ ($A\in \{N,E,D\}$) consists of 
$(p,q)$ Delannoy paths $L\in \mc{D}(p,q)$ whose last step is an $A$-step.  
Observe that, the map $\mc{D}_N(p,q)\rightarrow \mc{D}(p,q-1)$
that is defined by omitting the last (north) step is a weight preserving bijection.  
In a similar way, the map $\mc{D}_E(p,q)\rightarrow \mc{D}(p-1,q)$
that is defined by omitting the last (east) step is a weight preserving bijection.
Finally,  the map $\psi: \mc{D}_D(p,q)\rightarrow \mc{D}(p-1,q-1)$
that is defined by omitting the last (diagonal) step is a bijection and the 
weight of an element in its image, say $\psi(L)$, is $p+q-1$ times 
the weight of $L$. Putting these observations together we see that 
\begin{align*}
\sum_{L \in \mc{D}(p,q)} \omega(L) &= \sum_{L \in \mc{D}_N(p,q)} \omega(L) 
+ \sum_{L \in \mc{D}_E(p,q)} \omega(L)  +\sum_{L \in \mc{D}_D(p,q)} \omega(L) \\
&= \sum_{L \in \mc{D}(p,q-1)} \omega(L) 
+ \sum_{L \in \mc{D}(p-1,q)} \omega(L)  +\sum_{L \in \mc{D}(p-1,q-1)} (p+q-1)\omega(L).
\end{align*}
But this is the recurrence that we wanted to verify, therefore, the proof is finished. 
\end{proof}

Recall that a weighted $(p,q)$ Delannoy path
is a Delannoy path with certain weights on its diagonal steps;
on the $i$-th diagonal step, the weight 
is a positive integer between $1$ and $i-1$.
Recall also that the set of all weighted 
Delannoy paths is denoted by $\mc{P}(p,q)$. 

We are ready to prove our next main result.

\begin{Theorem}\label{T:anotherpaththeorem:body}
There is a bijection between the set of weighted $(p,q)$ Delannoy 
paths and the set of signed $(p,q)$-involutions.
In particular, we have $\alpha_{p,q} = \sum_{W\in \mc{P}(p,q)} 1$.
\end{Theorem}

\begin{proof}

Let $\pi = \pi^{(0)}=(i_1, j_1) \dots (i_k,j_k) l_1\dots l_{n-2k}$ 
be a signed involution from $I_{p,q}^\pm$. To construct the 
corresponding weighted path we proceed algorithmically 
as in the proof of Theorem~\ref{T:rec}.

First, we look at where $n=p+q$ appears in $\pi$. 
If it appears as a fixed point with a $+$ sign, then we draw a $E$-step 
between $(p,q)$ and $(p-1,q)$. If it appears as a fixed point with a $-$ sign, 
then we draw a $N$-step between $(p,q)$ and
$(p,q-1)$. In the these cases, removing $n$ from $\pi$ results in an involution, 
that we denote by $\pi^{(1)}$, in either $I_{k,p-1,q}^\pm$ or $I_{k,p,q-1}^\pm$. 
If $n$ appears as the second entry of one of the 2-cycles, say $(i_r,j_r)=(i,n)$ 
(for some $r$ and $i$), then we draw a 
$D$-step between $(p,q)$ and $(p-1,q-1)$, label it with $i$, and then 
we remove the two cycle $(i,n)$ from $\pi$ and 
reduce every number that is bigger than $i$ in $\pi$ by $-1$. 
Hence we obtain an element $\pi^{(1)}$ of $I_{k-1,p-1,q-1}^\pm$. 
We see from the proof of Theorem~\ref{T:rec} 
that this algorithm results in a bijection. 
\end{proof}

\vspace{.5cm}

{\em \textbf{Notation:}}
The bijection $\phi$ that is constructed in the proof of
Theorem~\ref{T:anotherpaththeorem:body} will be denoted by $\phi$, 
without giving any reference to the indices $p$ and $q$.

\begin{Example}

Let $\pi = (1,4)(3,8) 2^+ 5^+ 6^+ 7^- $. Then $p+q = 8$ and $p-q=2$, hence $p=5,q=3$.
Then $\phi(\pi)$ is the path that is in the last picture of Figure~\ref{F:last pic}.
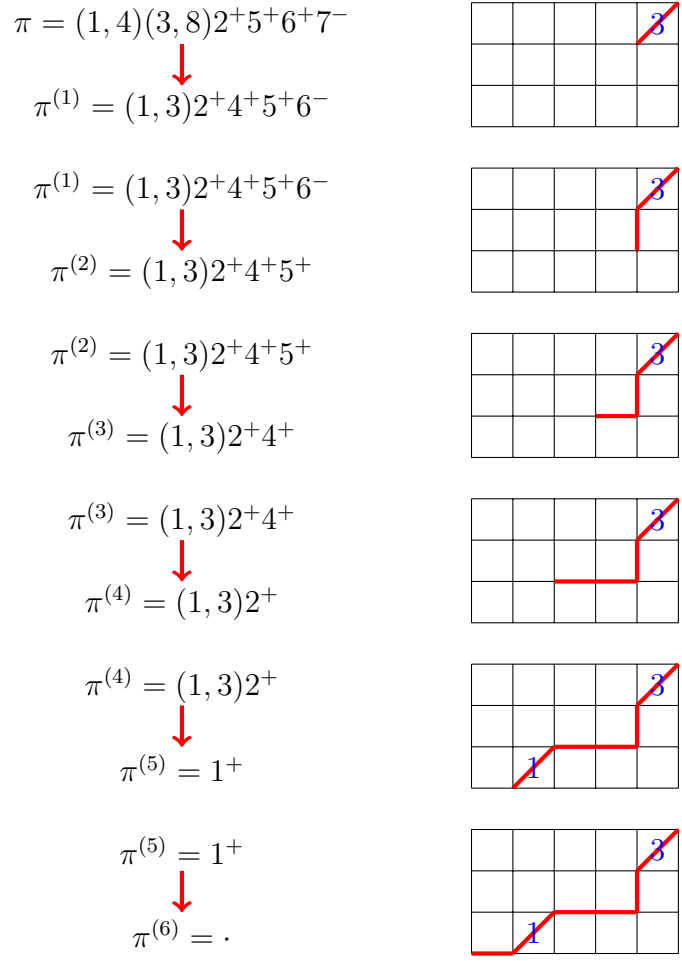
\begin{figure}[h]
\begin{center}
\begin{tikzpicture}[scale=.55]

\begin{scope}[xshift = -3.5cm]
\node at (0,2.5) {$\pi = (1,4)(3,8) 2^+ 5^+ 6^+ 7^- $};
\node at (0,.5) {$\pi^{(1)}= (1,3)2^+ 4^+ 5^+ 6^- $};
\draw[ultra thick,red, ->] (0,2) to (0,1);
\end{scope}

\begin{scope}[xshift = 3.5cm]
\draw (0,0) grid (5,3);
\draw[ultra thick,red] (5,3) to (4,2);
\node[blue] at (4.5,2.5) {$3$};

\end{scope}

\begin{scope}[xshift = -3.5cm, yshift=-4cm]
\node at (0,2.5) {$\pi^{(1)} = (1,3)2^+ 4^+ 5^+ 6^- $};
\node at (0,.5) {$\pi^{(2)} = (1,3)2^+ 4^+ 5^+ $};
\draw[ultra thick,red, ->] (0,2) to (0,1);
\end{scope}

\begin{scope}[xshift = 3.5cm,  yshift=-4cm]
\draw (0,0) grid (5,3);
\draw[ultra thick,red] (5,3) to (4,2);
\draw[ultra thick,red] (4,2) to (4,1);
\node[blue] at (4.5,2.5) {$3$};
\end{scope}

\begin{scope}[xshift = -3.5cm, yshift=-8cm]
\node at (0,2.5) {$\pi^{(2)} = (1,3)2^+ 4^+ 5^+ $};
\node at (0,.5) {$\pi^{(3)} = (1,3)2^+ 4^+ $};
\draw[ultra thick,red, ->] (0,2) to (0,1);
\end{scope}

\begin{scope}[xshift = 3.5cm,  yshift=-8cm]
\draw (0,0) grid (5,3);
\draw[ultra thick,red] (5,3) to (4,2);
\draw[ultra thick,red] (4,2) to (4,1);
\draw[ultra thick,red] (4,1) to (3,1);
\node[blue] at (4.5,2.5) {$3$};
\end{scope}

\begin{scope}[xshift = -3.5cm, yshift=-12cm]
\node at (0,2.5) {$\pi^{(3)} =(1,3)2^+ 4^+ $};
\node at (0,.5) {$\pi^{(4)} = (1,3)2^+ $};
\draw[ultra thick,red, ->] (0,2) to (0,1);
\end{scope}

\begin{scope}[xshift = 3.5cm,  yshift=-12cm]
\draw (0,0) grid (5,3);
\draw[ultra thick,red] (5,3) to (4,2);
\draw[ultra thick,red] (4,2) to (4,1);
\draw[ultra thick,red] (4,1) to (3,1);
\draw[ultra thick,red] (3,1) to (2,1);
\node[blue] at (4.5,2.5) {$3$};
\end{scope}

\begin{scope}[xshift = -3.5cm, yshift=-16cm]
\node at (0,2.5)  {$\pi^{(4)} = (1,3)2^+ $};
\node at (0,.5) {$\pi^{(5)} = 1^+$};
\draw[ultra thick,red, ->] (0,2) to (0,1);
\end{scope}

\begin{scope}[xshift = 3.5cm,  yshift=-16cm]
\draw (0,0) grid (5,3);
\draw[ultra thick,red] (5,3) to (4,2);
\draw[ultra thick,red] (4,2) to (4,1);
\draw[ultra thick,red] (4,1) to (3,1);
\draw[ultra thick,red] (3,1) to (2,1);
\draw[ultra thick,red] (2,1) to (1,0);
\node[blue] at (4.5,2.5) {$3$};
\node[blue] at (1.5,.5) {$1$};
\end{scope}

\begin{scope}[xshift = -3.5cm, yshift=-20cm]
\node at (0,2.5) {$\pi^{(5)} = 1^+$};
\node at (0,.5) {$\pi^{(6)} = \cdot$};
\draw[ultra thick,red, ->] (0,2) to (0,1);
\end{scope}

\begin{scope}[xshift = 3.5cm,  yshift=-20cm]
\draw (0,0) grid (5,3);
\draw[ultra thick,red] (5,3) to (4,2);
\draw[ultra thick,red] (4,2) to (4,1);
\draw[ultra thick,red] (4,1) to (3,1);
\draw[ultra thick,red] (3,1) to (2,1);
\draw[ultra thick,red] (2,1) to (1,0);
\draw[ultra thick,red] (1,0) to (0,0);
\node[blue] at (4.5,2.5) {$3$};
\node[blue] at (1.5,.5) {$1$};
\end{scope}

\end{tikzpicture}

\end{center}
\caption{Algorithmic construction of $\phi$.}
\label{F:last pic}
\end{figure}

\end{Example}

The covering relations of the weak order, equivalently 
the action of the Richardson-Springer monoid on $I_{p,q}^\pm$, 
as described in~\cite[Figure 2.5]{CMW}, are easy to express in 
terms of certain simple operations on the elements of $\mc{P}(p,q)$.
For the Bruhat order, the covering relations are described in
~\cite[Theorem 2.8]{Wyser16} and we are able to translate these 
relations into our language without difficulty. 
In Figure~\ref{F:LPaths}, we illustrated both of the 
Bruhat and the weak orders on $\mc{P}(2,2)$. The dashed lines
indicates the edges in the Bruhat order that are not present
in the weak order.

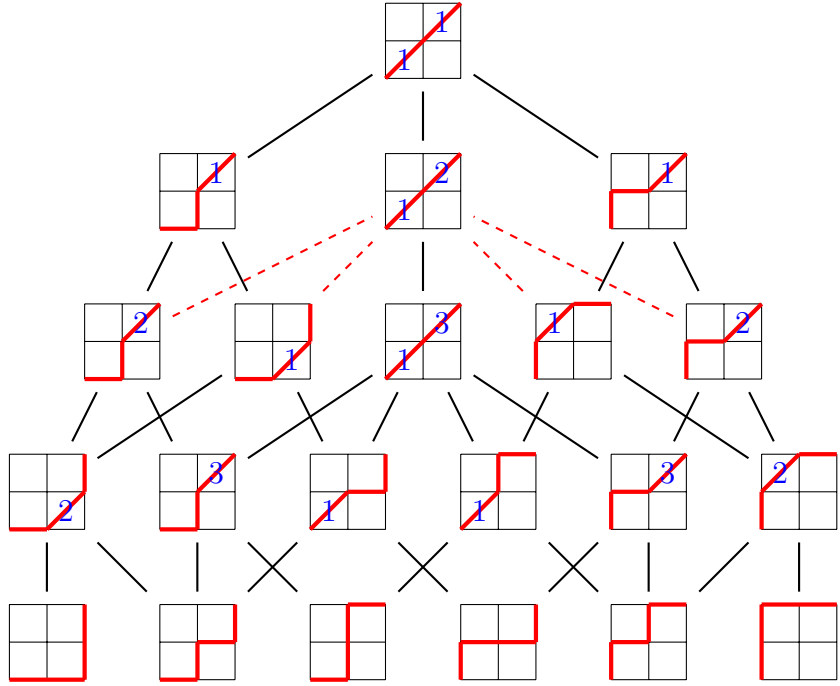
\begin{figure}[htp]
\begin{center}
\begin{tikzpicture}[scale=.2]
\node at (-25,0) (a1) {$\begin{tikzpicture}[scale=.5]
\draw (0,0) grid (2,2);
	\draw[ultra thick,red] (0,0) to (2,0);
	\draw[ultra thick,red] (2,0) to (2,2);
\end{tikzpicture}	$};
\node at (-15,0) (a2) {$\begin{tikzpicture}[scale=.5]
\draw (0,0) grid (2,2);
	\draw[ultra thick,red] (0,0) to (1,0);
	\draw[ultra thick,red] (1,0) to (1,1);
		\draw[ultra thick,red] (1,1) to (2,1);
		\draw[ultra thick,red] (2,1) to (2,2);
\end{tikzpicture}	$};
\node at (-5,0) (a3) {$\begin{tikzpicture}[scale=.5]
\draw (0,0) grid (2,2);
	\draw[ultra thick,red] (0,0) to (1,0);
		\draw[ultra thick,red] (1,0) to (1,2);
	\draw[ultra thick,red] (1,2) to (2,2);
\end{tikzpicture}	$};
\node at (5,0) (a4) {$\begin{tikzpicture}[scale=.5]
\draw (0,0) grid (2,2);
	\draw[ultra thick,red] (0,0) to (0,1);
	\draw[ultra thick,red] (0,1) to (2,1);
		\draw[ultra thick,red] (2,1) to (2,2);
\end{tikzpicture}	$};
\node at (15,0) (a5) {$\begin{tikzpicture}[scale=.5]
\draw (0,0) grid (2,2);
	\draw[ultra thick,red] (0,0) to (0,1);
		\draw[ultra thick,red] (0,1) to (1,1);
	\draw[ultra thick,red] (1,1) to (1,2);
	\draw[ultra thick,red] (1,2) to (2,2);
\end{tikzpicture}	$};
\node at (25,0) (a6) {$\begin{tikzpicture}[scale=.5]
\draw (0,0) grid (2,2);
	\draw[ultra thick,red] (0,0) to (0,2);
	\draw[ultra thick,red] (0,2) to (2,2);
\end{tikzpicture}	$};

\node at (-25,10) (b1) {$\begin{tikzpicture}[scale=.5]
\draw (0,0) grid (2,2);
	\draw[ultra thick,red] (0,0) to (1,0);
	\draw[ultra thick,red] (1,0) to (2,1);
	\draw[ultra thick,red] (2,1) to (2,2);
	\node[blue] at (1.5,.5) {$2$};
\end{tikzpicture}	$};
\node at (-15,10) (b2) {$\begin{tikzpicture}[scale=.5]
\draw (0,0) grid (2,2);
	\draw[ultra thick,red] (0,0) to (1,0);
	\draw[ultra thick,red] (1,0) to (1,1);
	\draw[ultra thick,red] (1,1) to (2,2);
	\node[blue] at (1.5,1.5) {$3$};
\end{tikzpicture}	$};
\node at (-5,10) (b3) {$\begin{tikzpicture}[scale=.5]
\draw (0,0) grid (2,2);
	\draw[ultra thick,red] (0,0) to (1,1);
	\draw[ultra thick,red] (1,1) to (2,1);
	\draw[ultra thick,red] (2,1) to (2,2);
	\node[blue] at (.5,.5) {$1$};
\end{tikzpicture}	$};
\node at (5,10) (b4) {$\begin{tikzpicture}[scale=.5]
\draw (0,0) grid (2,2);
	\draw[ultra thick,red] (0,0) to (1,1);
	\draw[ultra thick,red] (1,1) to (1,2);
	\draw[ultra thick,red] (1,2) to (2,2);
	\node[blue] at (.5,.5) {$1$};
\end{tikzpicture}	$};
\node at (15,10) (b5) {$\begin{tikzpicture}[scale=.5]
\draw (0,0) grid (2,2);
	\draw[ultra thick,red] (0,0) to (0,1);
	\draw[ultra thick,red] (0,1) to (1,1);
	\draw[ultra thick,red] (1,1) to (2,2);
	\node[blue] at (1.5,1.5) {$3$};
\end{tikzpicture}	$};
\node at (25,10) (b6) {$\begin{tikzpicture}[scale=.5]
\draw (0,0) grid (2,2);
	\draw[ultra thick,red] (0,0) to (0,1);
	\draw[ultra thick,red] (0,1) to (1,2);
	\draw[ultra thick,red] (1,2) to (2,2);
	\node[blue] at (.5,1.5) {$2$};
\end{tikzpicture}	$};

\node at (-20,20) (c1) {$\begin{tikzpicture}[scale=.5]
\draw (0,0) grid (2,2);
	\draw[ultra thick,red] (0,0) to (1,0);
	\draw[ultra thick,red] (1,0) to (1,1);
	\draw[ultra thick,red] (1,1) to (2,2);
	\node[blue] at (1.5,1.5) {$2$};
\end{tikzpicture}	$};
\node at (-10,20) (c2) {$\begin{tikzpicture}[scale=.5]
\draw (0,0) grid (2,2);
	\draw[ultra thick,red] (0,0) to (1,0);
	\draw[ultra thick,red] (1,0) to (2,1);
	\draw[ultra thick,red] (2,1) to (2,2);
		\node[blue] at (1.5,.5) {$1$};
\end{tikzpicture}	$};
\node at (0,20) (c3) {$\begin{tikzpicture}[scale=.5]
\draw (0,0) grid (2,2);
	\draw[ultra thick,red] (0,0) to (2,2);
	\node[blue] at (.5,.5) {$1$};
	\node[blue] at (1.5,1.5) {$3$};
\end{tikzpicture}	$};
\node at (10,20) (c4) {$\begin{tikzpicture}[scale=.5]
\draw (0,0) grid (2,2);
	\draw[ultra thick,red] (0,0) to (0,1);
	\draw[ultra thick,red] (0,1) to (1,2);
	\draw[ultra thick,red] (1,2) to (2,2);
		\node[blue] at (.5,1.5) {$1$};
\end{tikzpicture}	$};
\node at (20,20) (c5) {$\begin{tikzpicture}[scale=.5]
\draw (0,0) grid (2,2);
	\draw[ultra thick,red] (0,0) to (0,1);
	\draw[ultra thick,red] (0,1) to (1,1);
	\draw[ultra thick,red] (1,1) to (2,2);	
	\node[blue] at (1.5,1.5) {$2$};
\end{tikzpicture}	$};

\node at (-15,30) (d1) {$\begin{tikzpicture}[scale=.5]
\draw (0,0) grid (2,2);
	\draw[ultra thick,red] (0,0) to (1,0);
	\draw[ultra thick,red] (1,0) to (1,1);
	\draw[ultra thick,red] (1,1) to (2,2);
	\node[blue] at (1.5,1.5) {$1$};
\end{tikzpicture}	$};
\node at  (0,30) (d2) {$\begin{tikzpicture}[scale=.5]
\draw (0,0) grid (2,2);
	\draw[ultra thick,red] (0,0) to (2,2);
	\node[blue] at (.5,.5) {$1$};
	\node[blue] at (1.5,1.5) {$2$};
\end{tikzpicture}	$};
\node at  (15,30) (d3) {$\begin{tikzpicture}[scale=.5]
\draw (0,0) grid (2,2);
	\draw[ultra thick,red] (0,0) to (0,1);
	\draw[ultra thick,red] (0,1) to (1,1);
	\draw[ultra thick,red] (1,1) to (2,2);
	\node[blue] at (1.5,1.5) {$1$};
\end{tikzpicture}	$};

\node at (0,40) (e) {$\begin{tikzpicture}[scale=.5]
\draw (0,0) grid (2,2);
	\draw[ultra thick,red] (0,0) to (2,2);
	\node[blue] at (.5,.5) {$1$};
	\node[blue] at (1.5,1.5) {$1$};
\end{tikzpicture}	$};

\draw[-,  thick] (a1) to (b1);
\draw[-,  thick] (a2) to (b1);
\draw[-,  thick] (a2) to (b2);
\draw[-,  thick] (a2) to (b3);
\draw[-,  thick] (a3) to (b2);
\draw[-,  thick] (a3) to (b4);
\draw[-,  thick] (a4) to (b3);
\draw[-,  thick] (a4) to (b5);
\draw[-,  thick] (a5) to (b4);
\draw[-,  thick] (a5) to (b5);
\draw[-,  thick] (a5) to (b6);
\draw[-,  thick] (a6) to (b6);

\draw[-, thick] (b1) to (c1);
\draw[-,  thick] (b1) to (c2);
\draw[-, thick] (b2) to (c1);
\draw[-, thick] (b2) to (c3);
\draw[-,  thick] (b3) to (c2);
\draw[-, thick] (b3) to (c3);
\draw[-, thick] (b4) to (c3);
\draw[-, thick] (b4) to (c4);
\draw[-, thick] (b5) to (c3);
\draw[-, thick] (b5) to (c5);
\draw[-, thick] (b6) to (c4);
\draw[-, thick] (b6) to (c5);

\draw[-, thick] (c1) to (d1);
\draw[-, thick] (c2) to (d1);
\draw[dashed, thick, red] (c1) to (d2);
\draw[dashed, thick, red] (c2) to (d2);
\draw[-, thick] (c3) to (d2);
\draw[dashed, thick, red] (c4) to (d2);
\draw[dashed, thick, red] (c5) to (d2);
\draw[-, thick] (c4) to (d3);
\draw[-, thick] (c5) to (d3);

\draw[-, thick] (d1) to (e);
\draw[-, thick] (d2) to (e);
\draw[-, thick] (d3) to (e);
\end{tikzpicture}

\caption{The (weak) Bruhat order on weighted Delannoy paths.}
\label{F:LPaths}
\end{center}
\end{figure}

\section{Polynomial analogs}\label{S:Analogs}

In this section, we will consider various 
$t$-analogs of $\alpha_{p,q}$'s. 
In particular, we will investigate the length generating function 
of the (weak) Bruhat order on $\mc{P}(p,q)$.

\subsection{The weight generating function.}

Our first polynomial analog of the $\alpha_{p,q}$'s is defined as follows:
\begin{align}\label{A:definition of D}
D_{p,q}(t) := \frac{1}{t}\sum_{L \in \mc{D}(p,q)} t^{ \omega(\pi)}.
\end{align}
This is the generating function, up to a factor of $t$, 
for the weight function $\omega$ as in (\ref{A:weightofL}). 
It follows from the definitions that 
$D_{p,q}(t)$ obeys the recurrence 
\begin{align}\label{A:obeys1}
D_{p,q}(t) = D_{p-1,q}(t) + D_{p,q-1}(t) + t^{p+q-2} D_{p-1,q-1}(t^{p+q-1}).
\end{align}
In particular, we see from (\ref{A:obeys1}) that   
$$
\frac{\partial}{\partial t} (t D_{p,q}(t) ) |_{t=1} = \alpha_{p,q}.
$$
Obviously, $D_{p,q}(1)$ is nothing but the cardinality of the set 
$\mc{D}(p,q)$, that is the Delannoy number $D(p,q)$. 
The value at $t=0$ of $D_{p,q}(t)$ is easy to compute as well.
We will find it as the special case of a more general observation.

\vspace{.5cm}

Evaluating $D_{p,q}(t)$'s at other roots of unities also gives Delannoy numbers.
\begin{Proposition}
The value of the difference $D_{p,q}(t) - D_{p-1,q}(t) - D_{p,q-1}(t)$ at a $(p+q-1)$-th 
root of unity $\zeta$ is equal to $D(p-1,q-1)\zeta^{-1}$.
\end{Proposition}

\begin{proof}
It follows immediately from Definition (\ref{A:definition of D}) that 
$$
t(D_{p,q}(t) - D_{p-1,q}(t) - D_{p,q-1}(t)) = \sum_{P\in \mc{D}(p-1,q-1)} t^{(p+q-1) \omega(P)}.
$$
Therefore, evaluating both sides at $\zeta$ and then dividing by $\zeta$ gives 
$$
\frac{1}{\zeta} \sum_{P\in \mc{D}(p-1,q-1)} \zeta^{(p+q-1) \omega(P)} 
= \frac{1}{\zeta} \sum_{P\in \mc{D}(p-1,q-1)} 1^{ \omega(P)} = D(p-1,q-1) \zeta^{-1}.
$$
\end{proof}

Next, we are going to have a careful look at the coefficients of $D_{p,q}(t)$. 
It turns out they are always sums of products of binomial coefficients. 
Let $n\geq 1$ denote the degree of $D_{p,q}(t)$ and set 
$$
D_{p,q}(t) = a_0 + a_1 t + \cdots + a_n t^n \qquad (a_i\in \N).
$$
We start with the constant term, that is the value at $t=0$ of $\mc{D}_{p,q}(t)$. 
It is clear from our definition of the weight of a Delannoy path 
$L \in \mc{D}(p,q)$ that $\omega(L)=1$ if and only if $L$ 
has at most one diagonal step, which occurs as an initial diagonal step. 
Otherwise, the weight would be greater than 1. Consequently, 
\begin{align}\label{A:a_0}
a_0 = b_{p,q} + b_{p-1,q-1} = {p+q \choose q} + {p+q-2 \choose q-1 }.
\end{align}

Let $(a,b)$ and $(c,d)$ be two lattice points from the first quadrant 
in $\Z^2$ such that $a\leq c$ and $b\leq d$. 
We denote by $\mc{D}((a,b),(c,d))$ the set of all ``Delannoy 
paths'' that starts at $(a,b)$ and ends at $(c,d)$. In other words, 
$\mc{D}((a,b),(c,d))$ is the set of lattice paths that are obtained
from the Delannoy paths in $\mc{D}(c-a,d-b)$ by shifting their 
starting point to $(a,b)$.

Let $a_1 < \ldots < a_r \leq p$ and $b_1 <\ldots < b_r \leq q$ be two sequences
of nonnegative integers, denoted by $\text{a}$ and $\text{b}$, respectively. 
We define $\mc{D}_{\mt{a,b}}(p,q)$ as the subset of $\mc{D}(p,q)$ consisting 
of the lattice paths $L\in \mc{D}(p,q)$ with diagonal steps 
$L_i = ((a_i,b_i),(a_i+1,b_i+1))$ for $i=1,\dots,r$.  
Clearly, each element $L\in \mc{D}_{\mt{a,b}}(p,q)$ is a concatenation of 
$r+1$ lattice paths $L^{(1)},\dots, L^{(r)}$ each without a diagonal step.
More precisely, $L^{(i)} \in \mc{D}((a_i+1,b_i+1),(a_{i+1},b_{i +1}))$ 
and $L^{(i)}$ does not have any $D$-steps, for $i=0,\dots, r$. 
Here $(a_0,b_0)=(0,0)$ and $(a_{r+1},b_{r+1})=(p,q)$.
Therefore, the cardinality of $\mc{D}_{\mt{a,b}}(p,q)$ is given by 
\begin{align*}
\mc{D}_{\mt{a,b}}(p,q)= \prod_{i=0}^r | \mc{D}((a_i+1,b_i+1),(a_{i+1},b_{i +1})) | 
= \prod_{i=0}^r { a_{i+1} +b_{i+1} - a_i - b_i -2  \choose a_{i+1} - a_i-1 }.
\end{align*}
Note that the weight of any element $L \in \mc{D}_{\mt{a,b}}(p,q)$ is equal to 
\begin{align*}
\omega( L )= \prod_{i=0}^r ( a_{i} +b_{i} -1 ).
\end{align*}
Thus, by varying the number and the choice of the diagonal entries, 
we obtain a formula for $D_{p,q}(t)$.

\begin{Proposition}\label{P:explicit}
For all nonnegative integers $p$ and $q$, 
the explicit form of the weight generating polynomial $D_{p,q}(t)$ is given by
\begin{align}\label{A:A formula}
D_{p,q}(t) = \sum_{r=0}^{\min \{ q,p\}} 
\sum_{
\begin{subarray}{l}  
0\leq a_1 <\ldots <a_r \leq p  \\ 
0 \leq b_1 <\ldots < b_r \leq q
\end{subarray}}
\left( \prod_{i=0}^r { a_{i+1} +b_{i+1} - a_i - b_i -2  \choose a_{i+1} - a_i-1 }
 \right) t^{\prod_{i=0}^r ( a_{i} +b_{i} -1 )}.
\end{align}
\end{Proposition}
The proof of Proposition~\ref{P:explicit} follows from the previous
discussion so we omit it.

\begin{Remark}
In general, the polynomials $D_{p,q}(t)$ are not unimodal. 
\end{Remark}

\subsection{Length generating function.}

Our second $t$-analog has an algebro-geometric significance. 
Recall that the inclusion poset on the Borel orbit closures in 
$X=SL_{n}/S(GL_p\times GL_q)$ is the Bruhat order on signed
involutions $I_{p,q}^\pm$. This is a graded poset and its rank 
is equal to the length of its maximal element, which is 
$(1,n)\cdots ( q,n+1-q) (q+1)^+ \cdots (n-q)^+ $. 
The corresponding unique maximal element of 
$\mc{P}(p,q)$ is the path depicted in Figure~\ref{F:max}. 
Note that the closed Borel orbits in $X$ 
correspond to the weighted $(p,q)$ Delannoy paths 
without diagonal steps. In other words, the closed Borel orbits 
in $X$ correspond to the $(p,q)$ Grassmann paths.
See Remark~\ref{R:GrassmanPaths}.

\begin{figure}[htp]
\begin{center}
\begin{tikzpicture}[scale=.7]

\begin{scope}
\draw[dotted] (0,0) grid (9,4);

\node at (9.5,4.5) {$(p,q)$};
\node at (-.5,-.5) {$(0,0)$};
\node at (6,-.5) {$(p-q,0)$};
\draw[ultra thick,red] (9,4) to (7,2);
\draw[ultra thick,red,dotted] (7,2) to (6,1);
\draw[ultra thick,red] (6,1) to (5,0);
\draw[ultra thick,red] (5,0) to (4,0);
\draw[ultra thick,red,dotted] (4,0) to (1,0);
\draw[ultra thick,red] (1,0) to (0,0);
\node[blue] at (8.5,3.5) {$1$};
\node[blue] at (7.5,2.5) {$1$};
\node[blue] at (5.5,.5) {$1$};
\end{scope}

\end{tikzpicture}

\end{center}
\caption{The maximal element of the (weak) Bruhat order on $\mc{P}(p,q)$.}
\label{F:max}
\end{figure}
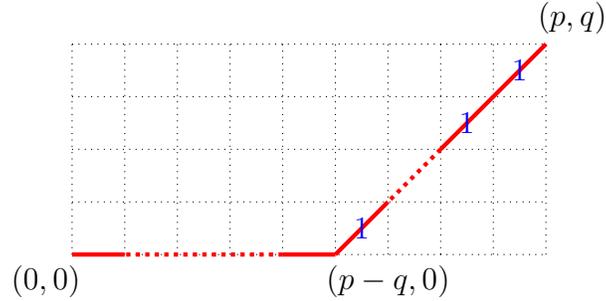

By Remark~\ref{R:it is observed} and (\ref{A:it is observed}) we see that 
\begin{align*}
\mt{rank} (I_{p,q}^\pm ) &=  \frac{ ((n-1) + \cdots + (n-q)) + (1+\cdots + (q-1))) + q}{2} \\ 
&=  \frac{nq - \frac{q(q+1)}{2} + \frac{q(q-1)}{2} + q }{2}\\
&= \frac{pq + q^2}{2}.
\end{align*}
Notice that the dimension of $X$ is $2pq$;
\begin{align*}
\dim X &= \dim SL_n - \dim S(GL_p\times GL_q)\\ 
&= (n^2-1) - (p^2 + q^2 - 1)\\
&= n^2 - p^2 - q^2 \\
&= 2pq.
\end{align*}
Therefore, the smallest possible dimension for a Borel orbit in $X$ is 
\begin{align*}
f_{min} (p,q) &:= \dim X - \text{rank}(I_{p,q}^\pm) \\
&= 2pq - \frac{pq + q^2}{2}\\
&= \frac{3pq- q^2}{2}.
\end{align*}

\begin{Remark}\label{R:explained}
If we denote the dimension of the Borel orbit in 
$X$ attached to $\pi \in I_{p,q}^\pm$ by $\dim \pi$, then 
$$
\dim \pi = \mathbb{L}(\pi) + f_{min}(p,q).
$$
Since, $f_{min}(p,q)$ is constant (relative to $p$ and $q$), 
the study of the function $\pi \mapsto \dim \pi$ is equivalent to study 
of the length function $\mathbb{L}(\cdot)$ on $I_{p,q}^\pm$. 
\end{Remark}
\begin{Definition}
The length generating function of $I_{p,q}^\pm$ is defined by
$$
E_{p,q}(t) := \sum_{\pi \in I_{p,q}^\pm} t^{\mathbb{L}(\pi)}.
$$
\end{Definition}

Our goal is to find a recurrence for $E_{p,q}(t)$. 
To this end we go back to our ideas in Section~\ref{S:Recurrence}.
Indeed, there is an important consequence of the proof of Theorem~\ref{T:rec}, 
where we essentially constructed a bijection  
$\psi_k=(\psi_k(+),\psi_k(-),\psi_k(1),\dots, \psi_k(n-1))$ from  
$$
I_{k,p-1,q}^\pm \times I_{k,p,q-1}^\pm \times  
\underbrace{I_{k-1,p-1,q-1}^\pm \times \cdots \times I_{k-1,p-1,q-1}^\pm}_{\text{ ($n-1$-copies)}}
$$
to $I_{k,p,q}^\pm$. 

In the light of Corollary~\ref{C:First recurrence}, 
we obtain the bijection $\psi=(\psi(+),\psi(-),\psi(1),\dots, \psi(n-1))$
\begin{align}\label{A:psi map}
\psi : I_{p,q-1}^\pm \times I_{p-1,q}^\pm \times  
\underbrace{I_{p-1,q-1}^\pm \times \cdots \times I_{p-1,q-1}^\pm}_{\text{ ($n-1$-copies)}} 
\longrightarrow I_{p,q}^\pm.
\end{align}

Next, we analyze the effect of maps $\psi(\pm)$ and $\psi(i)$, 
$i=1,\dots, n-1$ on the length of $\pi \in I_{p,q}^\pm$.

We know from Section~\ref{S:Preliminaries} that $\mathbb{L}(\pi)$ is equal to 
$(\ell(\pi)+k)/2$, where $k$ is the number of 2-cycles in $\pi$ and $\ell(\pi)$ 
is the number of inversions in $\pi$ viewed as a permutation. 
Thus, if $n$ is a fixed point of $\pi$, 
then removing it from $\pi$ has no effect on the length:
\begin{align}\label{A:pm case}
\mathbb{L}( \psi(\pm)^{-1} (\pi) )= \mathbb{L} (\pi).
\end{align}

For $\psi(i)$'s, it is more interesting; 
suppose $\pi$ has the standard form 
$$
\pi=(i_1,j_1)\cdots (i_k,j_k) c_1\dots c_{n-k}.
$$
If $n$ appears in the 2-cycle $(i_r,j_r)=(i,n)$, then by removing 
$(i_r,j_r)$ from $\pi$ we loose $n-i$ inversions of the form $n>j$
and we loose $n-i-1$ inversions of the form $j > i$. 
Moreover, we loose one 2-cycle. Therefore, 
\begin{align}\label{A:i case}
\mathbb{L}( \psi(i)^{-1}(\pi) ) = \frac{ \ell(\pi) + k - (2n-2i-1) - 1}{2} 
= \frac{\ell(\pi)+k}{2} - (n-i)= \mathbb{L}(\pi) - (n-i).
\end{align}

First, we partition $I_{p,q}^\pm$ into three disjoint sets 
$\psi(+)(I_{p-1,q}^\pm)$, $\psi(-)(I_{p,q-1}^\pm)$, and 
$\psi(i)(I_{p-1,q-1}^\pm)$ ($i=1,\dots, n-1$), then we reorganize 
the sums by using our observations (\ref{A:pm case}) and 
(\ref{A:i case}):
\begin{align*}
E_{p,q}(t) &= \sum_{\pi \in \psi(+)(I_{p-1,q}^\pm)} t^{\mathbb{L}(\pi)} 
+ \sum_{\pi \in \psi(-)(I_{p,q-1}^\pm)} t^{\mathbb{L}(\pi)} 
+ \sum_{ i=1}^{n-1} \sum_{\pi \in \psi(i)(I_{p-1,q-1}^\pm)} t^{\mathbb{L}(\pi)} \\
&= \sum_{\pi \in I_{p-1,q}^\pm} t^{\mathbb{L}(\pi)} + \sum_{\pi \in I_{p,q-1}^\pm} t^{\mathbb{L}(\pi)} 
+ \sum_{ i=1}^{n-1} \sum_{\pi \in I_{p-1,q-1}^\pm} t^{\mathbb{L}(\pi) + (n-i)} \\
&= E_{p-1,q}(t) + E_{p,q-1}(t)  + (t+ t^2 +\cdots + t^{n-1}) E_{p-1,q-1}(t). 
\end{align*}
The coefficient of the last term is equal to $[n]_t-1$, 
where $[n]_t$ stands for the $t$-analog of the natural number $n$:
$$
[n]_t := \frac{t^n-1}{t-1}.
$$
Thus we obtained the proof of the following result.
\begin{Theorem}
The length generating polynomials $E_{p,q}(t)$ (for $p,q\geq 1$)
satisfy the following recurrence:
$$
E_{p,q} (t) = E_{p-1,q}(t) + E_{p,q-1}(t) + ([q+p]_t-1) E_{p-1,q-1}(t).
$$
\end{Theorem}

\begin{Remark}
\begin{enumerate}
\item The polynomials $E_{p,q}(t)$ are unimodal. 
\item It appears that the sequence $(E_{n,n}(-1))_{n\geq 1}$ is 
the sequence of number of ``grand Motzkin paths'' of length $n$.
The sequence $(E_{n,n}(0))_{n\geq 1}$ is the sequence of number of 
``central binomial coeffients'' ${2n \choose n}$, $n\geq 1$.
(For other interpretations of these sequences, 
see The On-Line Encyclopedia of Integer Sequences, \url{https://oeis.org}.)
\item There is another closely related $t$-analogue. 
We define $\widetilde{E}_{p,q}(t)$'s by the recurrence
\begin{align*}
\widetilde{E}_{p,q}(t) = \widetilde{E}_{p-1,q}(t) + \widetilde{E}_{p,q-1}(t) + [q+p-1]_t \widetilde{E}_{p-1,q-1}(t).
\end{align*}
Similarly to $E_{p,q}(t)$, $\widetilde{E}_{p,q}(t)$ is a unimodal polynomial as well.  
Both of these families have interesting specializations.
\end{enumerate}
\end{Remark}

\subsection{Generating functions of $\gamma_{k,p,q}$'s.}

Recall that the numbers $\gamma_{k,p,q}$ are defined as 
the numbers of signed $(p,q)$-involutions with exactly $k$ 2-cycles and
in Theorem\ref{T:first main result:intro} we showed that 
$$
\gamma_{k,p,q}= \frac{(p+q)!}{(p-k)! (q-k)!} \frac{1}{2^k k !} \ \text{ for $k,p,q\geq 0.$} 
$$
In this subsection we study their generating function, namely 
$$
A_{p,q}(t) :=  \sum_{k=0}^{q} \gamma_{k,p,q} t^k.
$$

Our first observation is about an alternative description of $A_{p,q}(t)$'s. 
\begin{Proposition}
Let $p$ and $q$ be two nonnegative integers. Then 
$$
A_{p,q}(t) = \sum_{L \in \mc{P}(p,q)} t^{\# \text{of diagonal steps in $L$}}.
$$
\end{Proposition}
\begin{proof}
By definition, $A_{p,q}(t) = \sum_{k=0}^{\lfloor n/2 \rfloor} |I_{k,p,q}^\pm| t^k$. 
The map $\phi$ that is constructed in Theorem~\ref{T:anotherpaththeorem:body}
maps an element $\pi$ of $I_{k,q,p}^\pm$ to a weighted $(p,q)$ Delannoy path $L$ 
with exactly $k$ diagonal steps. The proof follows from the fact that 
$\phi$ is a bijection.
\end{proof}

Next, we will show that $A_{p,q}(t)$'s satisfy
a recurrence that is similar to $E_{p,q}(t)$'s. 

\begin{Proposition}
If $p$ and $q$ are two positive integers, then the following recurrence
relation holds true: 
\begin{equation*}
A_{p,q}(t) = A_{p,q-1}(t) + A_{p-1,q} (t)+  (p+q-1)t  A_{p-1,q-1}(t)
\end{equation*}
with initial conditions $A_{p,0}(t)=A_{0,q}(t)=1$ for all $p,q\geq 0$.
\end{Proposition}

\begin{proof}
The proof follows from the recurrence relation in Theorem~\ref{T:rec}.
\end{proof}

\begin{Remark}
$A_{p,q}(t)$'s seems to be unimodal as well.
\end{Remark}

\vspace{.25cm}
\textbf{Acknowledgements.} 
We thank Michael Joyce, Roger Howe, Luke Nelson, 
and Jeff Remmel.

\end{document}